\documentclass[10.5pt, a4paper]{amsart}
\usepackage{amssymb, amsmath, amscd, bm, amsthm, pdfpages, setspace, hyperref, mathtools, subcaption, graphicx, cite}

\def\@Rref#1{\hbox{\rm \ref{#1}}}
\def\Rref#1{\@Rref{#1}}
\theoremstyle{plain}
\newtheorem{thrm}{Theorem}[section]
\newtheorem{lmm}[thrm]{Lemma}

\newtheorem{prpstn}[thrm]{Proposition}
\theoremstyle{definition}
\newtheorem{dfntn}[thrm]{Definition}

\newtheorem{assumption}[thrm]{Assumption}
\theoremstyle{plain}

\newcommand{\re}{\mathop{\rm Re}\nolimits}
\newcommand{\im}{\mathop{\rm Im}\nolimits}

\newcommand{\ran}{\mathop{\rm ran}\nolimits}

\newcommand{\sr}{\mathop{\rm r}\nolimits}

\begin{document}

\title[Robustness of Polynomial Stability with Respect to Sampling]{Robustness of Polynomial Stability with Respect to Sampling}

\thispagestyle{plain}

\author{Masashi Wakaiki}
\address{Graduate School of System Informatics, Kobe University, Nada, Kobe, Hyogo 657-8501, Japan}
 \email{wakaiki@ruby.kobe-u.ac.jp}
 \thanks{This work was supported by JSPS KAKENHI Grant Number JP20K14362.}

\begin{abstract}
We provide a partially affirmative answer to the following
question on robustness of polynomial stability with respect to sampling:
``Suppose that a continuous-time state-feedback controller achieves the 
polynomial stability of the infinite-dimensional linear system. We apply
an idealized sampler and a zero-order hold to a feedback loop around the controller. Then, is the sampled-data system strongly stable
for all sufficiently small sampling periods? Furthermore,
is the polynomial decay of the continuous-time system transferred to
the sampled-data system under sufficiently fast sampling?''
The generator of the open-loop system is assumed to be 
a Riesz-spectral operator whose eigenvalues are not on the imaginary
axis but may approach it asymptotically.
We provide conditions for strong stability to be preserved under
fast sampling. Moreover, we
estimate the decay rate of the state of the sampled-data system
with a  smooth initial state and a sufficiently small sampling period.
\end{abstract}



\maketitle

\section{Introduction}
We study the robustness of polynomial stability with 
respect to sampling. To state our problem precisely, we 
consider the  following sampled-data system with sampling period $\tau >0$:
\begin{subequations}
	\label{eq:sampled_data_sys_intro}
	\begin{align}
		\dot x(t) &= Ax(t) + Bu(t),\quad t\geq 0;\qquad x(0) = x^0 \in X \\
		u(t) &= Fx(k\tau),\quad k\tau \leq t < (k+1)\tau,~k\in \mathbb{N}\cup \{0\},
		\label{eq:control_input_intro}
	\end{align}
\end{subequations}
where $A$ with domain $D(A)$ is the generator of a $C_0$-semigroup $(T(t))_{t \geq 0}$ on a Hilbert space $X$,
and
the control operator $B\colon \mathbb{C}\to X $ 
and the feedback operator $F\colon  X \to  \mathbb{C}$ are bounded 
linear operators.
We assume that the $C_0$-semigroup 
$(T_{BF}(t))_{t \geq 0}$ generated by $A+BF$ is polynomially stable
with parameter $\alpha >0$, which means that
$\sup_{t \geq 0}\|T_{BF}(t)\| < \infty$, the spectrum of $A+BF$ is contained in the open left half-plane, and for all $x \in D(A+BF) = D(A)$,
$\|T_{BF}(t)x\| = o(t^{-1/\alpha})$ as $t \to \infty$, i.e., 
for any $\varepsilon >0$, there exists $t_0 >0$ such that 
for all $t \geq t_0$,
\[
\|T_{BF}(t)x\| \leq \frac{\varepsilon}{t^{1/\alpha}}.
\]
By density of $D(A)$ in $X$, we see that under this assumption,
$(T_{BF}(t))_{t \geq 0}$ is strongly stable, that is, 
\[
\lim_{t\to \infty}\|T_{BF}(t)x^0\| = 0
\]
for all $x^0 \in X$. Intuitively, as the sampling period $\tau >0$ goes to zero,
the sampled-data control input  \eqref{eq:control_input_intro}
becomes closer to the continuous-time control input 
given by $u(t) = Fx(t)$ for $t \geq 0$.
Therefore, the following two questions arise:
\begin{enumerate}
	\renewcommand{\labelenumi}{\alph{enumi})}
	\item Is the sampled-data system \eqref{eq:sampled_data_sys_intro}
	with sufficiently small sampling period $\tau >0$ strongly stable
	in the sense that 
	\[
	\lim_{t \to \infty} x(t) = 0
	\]
	for every initial state $x^0 \in X$?
	\item 
	Does the state $x$ of the sampled-data 
	system \eqref{eq:sampled_data_sys_intro} decay polynomially
	for $x^0 \in D(A)$ and sufficiently small $\tau >0$ as
	the orbit $T_{BF}(t) x^0$?
\end{enumerate}
We provide a partially affirmative answer to these questions 
in this paper.
The effect of sampling on systems can be regarded as
a kind of structured perturbation. In this sense, 
the issue in the questions above is robustness analysis
of polynomial stability with respect to sampling.

For finite-dimensional linear systems, it is well known that 
the closed-loop stability is preserved under fast sampling.
However, the robustness of stability with respect to sampling
is not guaranteed for all 
infinite-dimensional linear systems; see \cite{Rebarber2002}.
It has been shown in \cite{Logemann2003,Rebarber2006} 
that  if $(T_{BF}(t))_{t \geq 0}$ is exponentially stable, that is,
there exist constants $M \geq 1$ and $\omega >0$ such that 
$\|T_{BF}(t)\| \leq M e^{-\omega t}$ for all $t \geq 0$, then
the sampled-data system also has the same property of 
exponential stability for a sufficiently small sampling period $\tau>0$.
Exponential stability is a strong property, which can be seen
from the fact that
exponential stability
is robust under small bounded perturbations and 
even some classes of unbounded perturbations as shown, e.g.,  in
\cite{Pritchard1989, Pandolfi1991}.
Exploiting the advantages of exponential stability,
the robustness analysis developed in \cite{Logemann2003,Rebarber2006}
allows unbounded control operators mapping the input space
into a space larger than the state space, called an extrapolation space.
On the other hand, the 
robustness of strong stability with respect to sampling
has been studied in \cite{Wakaiki2021SIAM}, where the control operator
needs an extra boundedness property related to the continuous spectrum
of $A$.
The reason for imposing
this  boundedness property is that 
strong stability is a rather delicate property that is highly sensitive
to perturbations; see \cite{Paunonen2014JDE,Paunonen2015Springer, Rastogi2020} 
for the robustness of
strong stability of $C_0$-semigroups 
(in the absence of  polynomial stability).

Exponential stability leads to uniformly quantified asymptotic
behaviors of semigroup orbits for all initial values
from the unit ball of the state space.
This is a desirable property from the viewpoint of many applications.
Nevertheless, exponential stability may be unachievable
in control problems, for example, involving 
wave equations or beam equations.
Although  strong stability can be achieved in some of those problems,
it is a qualitative notion of stability unlike exponential stability, and
we do not obtain any information on decay rates
of semigroup orbits from strong stability itself.
Polynomial stability is an important subclass of semi-uniform stability, which
lies between the above two
extreme types of semigroup stability, exponential stability
and strong stability, and guarantees semi-uniform
decay rates for semigroup orbits with 
initial values in the domain of the generator. 
Various results on polynomial stability, and more generally 
semi-uniform stability, have been obtained such as
characterizations of decay rates by resolvent estimates on the imaginary axis
$i\mathbb{R}$ \cite{Liu2005PDR, 
	Batkai2006, Batty2008,Borichev2010,Rozendaal2019} and
robustness to perturbations \cite{Paunonen2011,Paunonen2012SS,Paunonen2013SS, Paunonen2014OM, Rastogi2020}.
We also refer to \cite{Chill2020} for an overview of semi-uniform stability.
A discrete version of semi-uniform stability has been
investigated in the context of the quantified Katznelson-Tzafriri theorem \cite{Seifert2015, Seifert2016, Cohen2016, Ng2020} (see also the survey article \cite{Batty2022Survey})
and the Cayley transform of a semigroup generator \cite{Wakaiki2021JEE}.
However, to the author's knowledge, robustness
analysis with respect to sampling has not been well established for
polynomial stability.

To study the robustness of polynomial stability with respect to sampling,
this paper continues and expands the robustness analysis
developed in \cite{Wakaiki2021SIAM}.
We assume as in \cite{Wakaiki2021SIAM} that $A$ is a Riesz-spactral operator given
by
\begin{equation*}
	A x = 
	\sum_{n=1}^\infty \lambda_n \langle x , \psi_n \rangle \phi_n
\end{equation*}
with domain
\begin{equation*}
	D(A) = \left\{
	x \in X: \sum_{n=1}^\infty |\lambda_n|^2\, |\langle x, \psi_n \rangle|^2 < \infty
	\right\},
\end{equation*}
where $(\lambda_n)_{n \in \mathbb{N}}$ are distinct complex numbers
not on $ i \mathbb{R}$, 
$(\phi_n)_{n \in \mathbb{N}}$ forms a Riesz basis in $X$, and 
$(\psi_n)_{n \in \mathbb{N}}$ is a biorthogonal 
sequence to $(\phi_n)_{n \in \mathbb{N}}$; see Section~\ref{sec:RS_operator}
for the details of Riesz-spactral operators.
We restrict our attention to the situation where only a finite number of
the eigenvalues $(\lambda_n)_{n \in \mathbb{N}}$ are in the set 
\[
\Omega_{\rm a} \coloneqq 
\left\{
\lambda \in \mathbb{C} : \re \lambda > -\omega
\right\}
\cap
\left(
\mathbb{C} \setminus
\left\{
\lambda \in \mathbb{C} \setminus \mathbb{R} : \re \lambda \leq
\frac{- \Upsilon}{|\im \lambda|^\alpha}
\right\}
\right)
\]
for some $\omega,\alpha,\Upsilon >0$.
In contrast, it is assumed in \cite{Wakaiki2021SIAM} that 
the set
\[
\Omega_{\rm b} \coloneqq 
\left\{
\lambda \in \mathbb{C} \setminus \{
0
\}: \re\lambda > - \omega,~| \arg \lambda | < \pi /2 + \vartheta
\right\}
\]
contains only finitely many eigenvalues for some $\omega>0$ and 
$0< \vartheta \leq \pi/2$.
Figure~\ref{fig:Omega} illustrates the sets $\Omega_{\rm a}$ and $\Omega_{\rm b}$.
In our setting,  the continuous spectrum of $A$ has empty
intersection with 
$i\mathbb{R}$ unlike
the setting of \cite{Wakaiki2021SIAM}, but
the spectrum of $A$ may approach $i\mathbb{R}$ asymptotically.
In other words, the resolvent of $A$ 
has a {\em singularity at zero} in \cite{Wakaiki2021SIAM},
whereas, loosely speaking, the resolvent
restricted to $i\mathbb{R}$ has a {\em singularity at infinity} in this study because
the resolvent grows to infinity on $i\mathbb{R}$.
Therefore, the type of non-exponential stability we consider in this paper 
is 
different from that in \cite{Wakaiki2021SIAM}.
It has been shown in \cite{Wakaiki2021SIAM} that 
only strong stability is preserved under fast sampling. Here
we investigate the quantitative behavior of
the state of the sampled-data system in addition to strong stability.

\begin{figure}
	\centering
	\subcaptionbox{Set $\Omega_{\rm a}$ considered in this paper.
		\label{fig:Omega1}}
	[.49\linewidth]
	{\includegraphics[width = 5cm,clip]{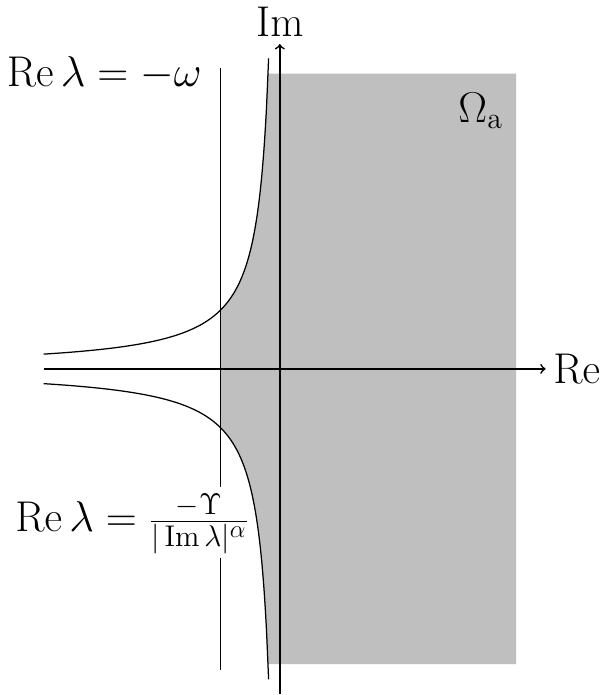}} 
	\subcaptionbox{Set $\Omega_{\rm b}$ considered in \cite{Wakaiki2021SIAM}.
		\label{fig:Omega2}}
	[.49\linewidth]
	{\includegraphics[width = 5cm,clip]{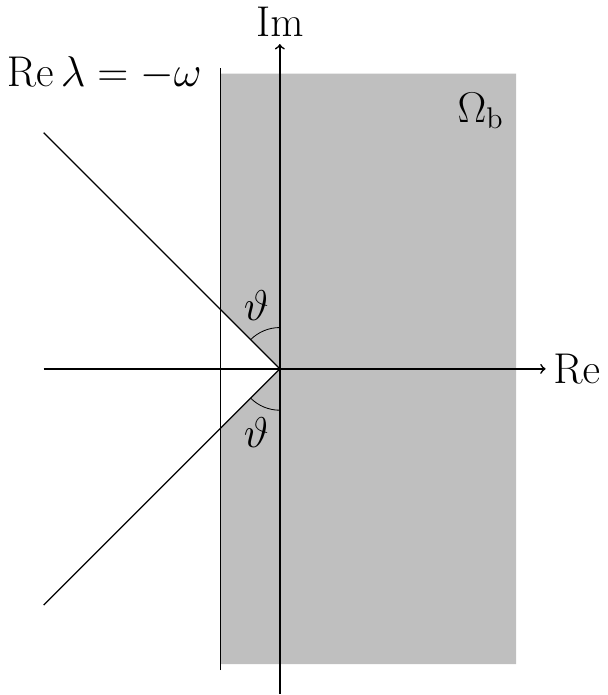}}
	\caption{Comparison of sets that contain only finitely many
		eigenvalues. \label{fig:Omega}}
\end{figure}

Another important difference from \cite{Wakaiki2021SIAM} is an assumption on
the control operator $B$ and the feedback operator $F$.
Let $b,f \in X$ and let
$B,F$ be written as $Bu = bu$ for 
all $u \in \mathbb{C}$ and $Fx = \langle x, f \rangle$
for all $x \in X$. In this paper, we assume that 
\begin{align*}
	b &\in  \mathcal{D}^\beta \coloneqq 
	\left\{ x \in X:
	\sum_{n=1}^{\infty} |\lambda_n|^{2\beta}\,
	|\langle x, \psi_n \rangle |^2 < \infty\right\}\\
	f &\in \mathcal{D}^\gamma_* \coloneqq 
	\left\{ x \in X:
	\sum_{n=1}^{\infty} |\lambda_n|^{2\gamma}\,
	|\langle x,\phi_n  \rangle |^2 < \infty\right\},
\end{align*}
where $\beta,\gamma \geq 0$ satisfy one of the following 
conditions: (i) $\beta$ and $\gamma$ are integers and $\beta + \gamma \geq \alpha$; or (ii) $\beta + \gamma > \alpha$.
On the other hand, it is assumed in \cite{Wakaiki2021SIAM}
that
\[
b \in D(A^{-1}) = \left\{ x \in X:
\sum_{n=1}^{\infty} \frac{
	|\langle x, \psi_n \rangle |^2}{|\lambda_n|^2} < \infty\right\}
\]
and $f \in X$.
Under
the assumption we make in this paper,
$B$ and $F$
have the parameters $\beta$ and $\gamma$ for design flexibility,
which increases the applicability
of the proposed robustness analysis.

The
bounded linear operator $\Delta(\tau)$ on $X$ defined by
\[
\Delta(\tau) \coloneqq T(\tau) + \int^{\tau}_0 T(s)BFds
\]
plays a key role in the analysis of robustness with respect to sampling. In fact, the sampled-data system \eqref{eq:sampled_data_sys_intro}
is strongly stable if and only if
the discrete semigroup $(\Delta(\tau)^k)_{k \in \mathbb{N}}$
is strongly stable, i.e.,  
\[
\lim_{k \to \infty} \|\Delta(\tau)^k x^0\| = 0
\]
for all $x^0 \in X$.
In \cite{Wakaiki2021SIAM}, the sufficient condition 
for strong stability obtained in
the Arendt-Batty-Lyubich-V\~u theorem~\cite{Arendt1998,Lyubich1988} is used in order to show that 
$(\Delta(\tau)^k)_{k \in \mathbb{N}}$ is strongly stable.
This sufficient condition requires that 
the intersection of the spectrum of $\Delta(\tau)$ and the unit circle
be countable, but 
the system we consider does not have this property in general.
Instead of the Arendt-Batty-Lyubich-V\~u theorem, we here employ
the characterization of strong stability by an integral condition on
resolvents developed in \cite{Tomilov2001}.

Let $0 < \delta \leq \alpha/2$, where 
$\alpha >0$ is the constant for the set $\Omega_{\rm a}$.
We give an integral condition on resolvents under which
the orbit $\Delta(\tau)^k x^0$ with
$x^0 \in \mathcal{D}^{\delta}$ satisfies
\[
\|\Delta(\tau)^k x^0\| = 
\begin{cases}
	o(k^{-\delta/\alpha}) & \text{if $0 < \delta < \alpha/2$} \vspace{5pt}\\
	o \left(
	\sqrt{
		\dfrac{\log k}{k}
	}
	\right) & \text{if $\delta = \alpha/2$}
\end{cases}
\]
as $k \to \infty$. Using this integral condition, we  
show that the state $x$
of the sampled-data system \eqref{eq:sampled_data_sys_intro}
with sufficiently small sampling period $\tau >0$
satisfies
\begin{equation}
	\label{eq:decay_estimate_intro}
	\|x(t)\| = 
	\begin{cases}
		o(t^{-\delta/\alpha}) & \text{if $0 < \delta< \alpha/2$} \vspace{5pt}\\
		o \left(
		\sqrt{
			\dfrac{\log t}{t}
		}
		\right) & \text{if $\delta  = \alpha/2$}
	\end{cases}
\end{equation}
as $t\to \infty$
for every initial state $x^0 \in \mathcal{D}^{\delta}$, provided that
$\delta \leq 1$ or $\beta \geq \alpha$.
Considering the open-loop case $F=0$, we see that
$t^{-\delta/\alpha}$ in the estimate \eqref{eq:decay_estimate_intro} 
cannot be replaced by functions with better decay rates.
It is still unknown whether 
the logarithmic factor $\sqrt{\log t}$ in the case $\delta = \alpha/2$ may be removed.

The paper is organized as follows.
Section~\ref{sec:preliminaries} contains preliminaries
on polynomial stability of $C_0$-semigroups and 
Riesz-spectral operators.
In Section~\ref{sec:SDS}, we present the main result
and introduce the discretized system for its proof.
Section~\ref{sec:resolvent_cond} is devoted to
resolvent conditions for stability.
To apply these conditions to the discretized system,
we investigate the spectrum of $\Delta(\tau)$ in Section~\ref{sec:spectrum}.
Section~\ref{sec:resolvent_discretized_sys} completes the proof 
of the main result with the help of the resolvent conditions for stability.
To illustrate the theoretical result, 
we study a wave equation in Section~\ref{sec:wave_eq}.
The conclusion is given in Section~\ref{sec:conclusion}.

\subsection*{Notation and terminology}
We denote by $\mathbb{N}_0$  the set of 
non-negative integers.
For $\omega \in \mathbb{R}$ and $r>0$, we write
\begin{align*}
	\mathbb{C}_\omega &\coloneqq 
	\{
	\lambda \in \mathbb{C}: \re \lambda > \omega
	\}\\
	\mathbb{D}_r &\coloneqq 
	\{
	\lambda \in \mathbb{C}: |\lambda| < r
	\} \\
	\mathbb{E}_r &\coloneqq 
	\{
	\lambda \in \mathbb{C}: |\lambda| > r
	\}.
\end{align*}
Note that we denote
the open right half-plane by
$\mathbb{C}_0$, while $\mathbb{C}^+$ and $\mathbb{C}_+$ are commonly
used in the literature.
For $\alpha ,\Upsilon >0$, 
define 
\[
\Omega_{\alpha,\Upsilon} \coloneqq 
\left\{
\lambda \in \mathbb{C}\setminus \mathbb{R} : \re \lambda \leq  \frac{-\Upsilon}{|\im \lambda|^\alpha}
\right\}.
\]
The closure of a subset $\Omega$ of $\mathbb{C}$ and
the complex conjugate of $\lambda \in \mathbb{C}$ are
denoted by $\overline{\Omega}$ and $\overline{\lambda}$, respectively.
For real-valued functions $f,g$ on $J \subset \mathbb{R}$, we write
\[
f(t) = O\big(g(t)\big)\qquad (t \to \infty)
\] 
if
there exist $M>0$ and $t_0 \in J$ such that 
$f(t) \leq Mg(t)$ for all $t \geq t_0$, and
similarly,
\[
f(t) = o\big(g(t)\big)\qquad (t \to \infty)
\]
if
for any $\varepsilon >0$, there exists $t_0 \in J$ such that 
$f(t) \leq \varepsilon g(t)$ for all $t \geq t_0$.

Let $X$ and $Y$ be Banach spaces. For a linear operator $A\colon X\to Y$,
we denote by $D(A)$ and $\ran (A)$ the domain and the range of $A$,
respectively.
The space of all bounded linear operators from $X$ to $Y$ is denoted by $\mathcal{L}(X,Y)$,
and we write $\mathcal{L}(X) \coloneqq \mathcal{L}(X,X)$.
For a linear operator $A\colon  D(A) \subset X \to X$, we denote by 
$\sigma(A)$ and $\rho(A)$ the spectrum and
the resolvent set of $A$, respectively.
We write $R(\lambda,A) \coloneqq (\lambda I - A)^{-1}$
for $\lambda \in \rho (A)$.
For a subset $S$ of $X$ and a linear operator 
$A\colon  D(A) \subset X \to Y$, we denote by $A|_S$ 
the restriction of $A$ to $S$, i.e., $A|_S x = Ax$ with domain
$D(A|_S) \coloneqq D(A) \cap S$.	

A $C_0$-semigroup $(T(t))_{t \geq 0}$ on a Banach space $X$ 
is called {\em uniformly bounded} if
$\sup_{t \geq 0} \|T(t)\| < \infty$ and
{\em strongly stable} if $\lim_{t \to \infty} T(t)x = 0$ for all $x \in X$.
By a {\em discrete semigroup} on $X$, we mean a family $(\Delta^k)_{k \in \mathbb{N}}$
of operators, where $\Delta \in \mathcal{L}(X)$.
A discrete semigroup 
$(\Delta^k )_{k \in \mathbb{N}}$ on $X$ is called 
{\em power bounded}  
if $\sup_{k \in \mathbb{N}}\|\Delta^k\| < \infty$ and
{\em strongly stable} if $\lim_{k \to \infty} \Delta^k x = 0$
fo all $x \in X$.

An inner product on a Hilbert space is denoted by $\langle \cdot, \cdot \rangle$.
For Hilbert spaces $Z$ and $W$, let $A^*$ denote
the Hilbert space adjoint of a densely
defined linear operator $A\colon D(A) \subset Z \to W$.

\def\theenumi{\alph{enumi})}
\def\labelenumi{\theenumi}
\def\theenumii{(\roman{enumii})}
\def\labelenumii{\theenumii}
\section{Preliminaries}
\label{sec:preliminaries}
In this section, we review the definition and 
some important properties of
polynomially stable $C_0$-semigroups and 
Riesz-spectral operators.

\subsection{Polynomially stable $C_0$-semigroups}
We start by recalling the definition of polynomially stable $C_0$-semigroups.
In Definition~3.2 of \cite{Batkai2006}, 
polynomial stability of $C_0$-semigroups does not include
uniform boundedness,
but here we define polynomial stability to include
uniform boundedness as in
Definition~1.2 of semi-uniform stability in \cite{Batty2008}.
\begin{dfntn}
	A $C_0$-semigroup $(T(t))_{t\geq 0}$ on a Banach space $X$
	generated by $A$ is {\em polynomially stable with parameter $\alpha >0$}
	if the following three conditions are satisfied:
	\begin{enumerate}
		\item
		$(T(t))_{t\geq 0}$ is uniformly bounded;
		\item $i \mathbb{R} \subset \rho(A)$; and
		\item $\|T(t)A^{-1}\| = O(t^{-1/\alpha})$ as $t\to \infty$.
	\end{enumerate}
\end{dfntn}

Let $A$ be the generator of a uniformly bounded $C_0$-semigroup $(T(t))_{t\geq 0}$
on a Hilbert space. Then $-A$ and $-A^*$ are
sectorial in the sense of Chapter~2 of \cite{Haase2006}.
In particular, if $(T(t))_{t\geq 0}$ is a polynomially stable $C_0$-semigroup
on a Hilbert space, then
$A$ and $A^*$ are invertible, and hence 
the fractional powers $(-A)^{\alpha}$ and $(-A^*)^{\alpha}$
are well defined for all $\alpha \in \mathbb{R}$.
We refer, e.g., to Chapter~3 of \cite{Haase2006} and 
Section~II.5.3 of \cite{Engel2000} for the details of 
fractional powers.

We use the following characterizations for polynomial decay  
of a $C_0$-semigroup on a Hilbert space.
The proof can be found in Lemma~2.3 and 
Theorem~2.4 of \cite{Borichev2010}. See also
Lemma~2.3 of \cite{Wakaiki2021JEE} for the result on 
the decay rate of an individual orbit.
\begin{thrm}
	\label{thm:decay_charac}
	Let $(T(t))_{t\geq 0}$ be a uniformly bounded $C_0$-semigroup
	on a Hilbert space $X$ with generator 
	$A$ such that $i \mathbb{R} \subset \rho(A)$. 
	For fixed $\alpha,\delta >0$, the following statements are equivalent:
	\begin{enumerate}
		\item
		$\|T(t)A^{-1}\| = O(t^{-1/\alpha})$ as $t\to \infty$. \vspace{2pt}
		\item[a')]
		$\|T(t)(-A)^{-\delta}\| = O(t^{-\delta/\alpha})$ as $t\to \infty$. \vspace{2pt}
		\item 
		$\|T(t)(-A)^{-\delta}x\| = o(t^{-\delta/\alpha})$ as $t\to \infty$ 
		for all $x \in X$. \vspace{2pt}
		\item
		$\|R(is,A)\| = O(|s|^{\alpha})$ as $|s| \to \infty$.
		\vspace{2pt}
		\item
		$\displaystyle 
		\sup_{\lambda \in \overline{\mathbb{C}_0}}
		\|R(\lambda,A)(-A)^{\alpha}\| < \infty$. 
	\end{enumerate}
\end{thrm}

The following estimate given in 
Lemma~4 of \cite{Paunonen2014OM} is useful in the robustness analysis of polynomial stability. 
\begin{lmm}
	\label{lem:resolvent_estimate}
	Let $A$ be the generator of a polynomially stable $C_0$-semigroup
	with parameter $\alpha >0$ 
	on a Hilbert space $X$.
	Let $\beta,\gamma \geq 0$ satisfy $\beta+\gamma \geq \alpha$
	and let $U$ be a Banach space.
	There exists a constant $M \geq 1$ such that 
	if $B \in \mathcal{L}(U,X)$ and $F \in \mathcal{L}(X,U)$
	satisfy $\ran (B) \subset D((-A)^\beta)$ and 
	$\ran (F^*) \subset D((-A^*)^\gamma)$, then
	\[
	\|FR(\lambda,A)B\| \leq M \|(-A)^{\beta}B\|\, \|(-A^*)^\gamma F^*\|
	\]
	for all $\lambda \in \overline{\mathbb{C}_0}$.
\end{lmm}

\subsection{Riesz-spectral operators}
\label{sec:RS_operator}
Next we recall the definition of Riesz-spectral operators and 
briefly state their most relevant properties. We refer the reader to
Section~3.2 of \cite{Curtain2020}, Section~2.4 of \cite{Tucsnak2009}, and Chapter~2 of \cite{Guo2019book} 
for more details.
\begin{dfntn}[Definition 3.2.6 of \cite{Curtain2020}]
	\label{def:RSO}
	Let $X$ be a Hilbert space and let
	$A\colon D(A) \subset X \to X$ be a closed linear operator with
	simple eigenvalues $(\lambda_n)_{n \in \mathbb{N}}$ and
	corresponding eigenvectors $(\phi_n)_{n \in \mathbb{N}}$. 
	We say that $A$ is a {\em Riesz-spectral operator} if
	the following two conditions are satisfied:
	\begin{enumerate}
		\item
		$(\phi_n)_{n \in \mathbb{N}}$ is a Riesz basis; and
		\item
		The set of eigenvalues $\{\lambda_n : n \in \mathbb{N} \}$
		has at most finitely many accumulation points.
	\end{enumerate}
\end{dfntn}

Let $A$ be a Riesz-spectral operator on a Hilbert space $X$
with simple eigenvalues 
$(\lambda_n)_{n \in \mathbb{N}}$ and 
corresponding eigenvectors $(\phi_n)_{n \in \mathbb{N}}$. 
Let $(\psi_n)_{n \in \mathbb{N}}$ be the eigenvectors 
of the adjoint $A^*$ corresponding to the eigenvalues 
$(\overline{\lambda_n})_{n \in \mathbb{N}}$.
Then
$(\psi_n)_{n \in \mathbb{N}}$ can be suitable scaled so that
$(\phi_n)_{n \in \mathbb{N}}$ and $(\psi_n)_{n \in \mathbb{N}}$
are {\em biorthogonal}, i.e.,
\[
\langle \phi_n, \psi_m \rangle = 
\begin{cases}
	1 & \text{if $n=m$} \\
	0 & \text{otherwise}.
\end{cases}
\]
A sequence biorthogonal to a Riesz basis in $X$ 
is unique and also forms a Riesz basis in $X$.
Throughout this paper, we set the sequence $(\psi_n)_{n \in \mathbb{N}}$
of the eigenvectors of
the adjoint $A^*$ so that $(\psi_n)_{n \in \mathbb{N}}$
are biorthogonal to $(\phi_n)_{n \in \mathbb{N}}$.
Every $x \in X$ can be represented uniquely by
\[
x = \sum_{n=1}^{\infty} \langle x, \psi_n \rangle \phi_n = 
\sum_{n=1}^{\infty} \langle x, \phi_n \rangle \psi_n.
\]
Moreover,
there exist constants $M_{\rm a}, M_{\rm b} >0$ such that
for all  $x \in X$,
\begin{align*}
	M_{\rm a} \sum_{n=1}^\infty |\langle x, \psi_n \rangle |^2 \leq \|&x\|^2 \leq 
	M_{\rm b} \sum_{n=1}^\infty |\langle x, \psi_n \rangle |^2
	\\
	\frac{1}{M_{\rm b}}\sum_{n=1}^\infty |\langle x, \phi_n \rangle |^2 \leq \|&x\|^2 \leq 
	\frac{1}{M_{\rm a}} \sum_{n=1}^\infty |\langle x, \phi_n \rangle |^2.
\end{align*}
We shall frequently use these inequalities without comment. 

The Riesz-spectral operator $A$ has the following representation:
\begin{equation}
	\label{eq:RS_operator}
	A x = 
	\sum_{n=1}^\infty \lambda_n \langle x , \psi_n \rangle \phi_n,\quad 
	x \in D(A)
\end{equation}
with domain
\begin{equation*}
	D(A) = \left\{
	x \in X: \sum_{n=1}^\infty |\lambda_n|^2\, |\langle x, \psi_n \rangle|^2 < \infty
	\right\}.
\end{equation*}
The spectrum of the Riesz-spectral operator $A$ is the closure of its point spectrum, that is,
$\sigma(A) = \overline{\{\lambda_n:n \in \mathbb{N}\}}$.
For $\lambda \in \rho(A)$, the resolvent $R(\lambda,A)$ is given by
\begin{equation}
	\label{eq:RS_resolvent}
	R(\lambda,A)x = \sum_{n=1}^\infty \frac{1}{\lambda- \lambda_n}
	\langle x , \psi_n \rangle \phi_n,\quad 
	x \in X.
\end{equation}
The Riesz-spectral operator $A$ generates a $C_0$-semigroup on $X$
if and only if
$\sup_{n \in \mathbb{N}} \re \lambda_n < \infty$, and the $C_0$-semigroup $(T(t))_{t\geq 0}$ generated by $A$ can be written as
\[
T(t)x = \sum_{n=1}^\infty e^{t\lambda_n } \langle x , \psi_n \rangle \phi_n
\]
for all $x \in X$ and $t \geq 0$.

The adjoint $A^*$ is also a Riesz-spectral operator and is 
represented as 
\[
A^* x = \sum_{n=1}^\infty \overline{\lambda_n} \langle 
x, \phi_n
\rangle \psi_n,\quad 
x \in D(A^*)
\]
with domain
\[
D(A^*) = \left\{
x \in X :
\sum_{n=1}^\infty 
\left|
\lambda_n
\right|^2 \,  \left|
\langle x, \phi_n \rangle
\right|^2 < \infty
\right\}.
\]
Moreover, the $C_0$-semigroup generated by $A^*$ is given by 
$(T(t)^* )_{t\geq 0}$.

To
make assumptions on the ranges of the control operator $B\in \mathcal{L}(\mathbb{C},X)$ and
the adjoint of the feedback operator $F\in\mathcal{L}(X,\mathbb{C})$,
we use the following subsets with parameters $\beta,\gamma \geq 0$:
\begin{align*}
	\mathcal{D}^{\beta} &\coloneqq
	\left\{ x \in X:
	\sum_{n=1}^{\infty} |\lambda_n|^{2\beta}\,
	|\langle x, \psi_n \rangle |^2 < \infty\right\}\\
	\mathcal{D}_*^{\gamma} &\coloneqq
	\left\{ x \in X:
	\sum_{n=1}^{\infty} |\lambda_n|^{2\gamma}\,
	|\langle x,\phi_n  \rangle |^2 < \infty\right\}.
\end{align*}

\section{Stability of sampled-data systems}
\label{sec:SDS}
In this section, we 
present the system under consideration and state the main result. We also introduce the discretized system as the first step of its proof.

\subsection{Main result}
Let $X$ be a Hilbert space, and
consider the following sampled-data system with 
state space $X$ and input space $\mathbb{C}$: 
\begin{subequations}
	\label{eq:sampled_data_sys}
	\begin{align}
		\dot x(t) &= Ax(t) + Bu(t),\quad t\geq 0;\qquad x(0) = x^0 \in X \\
		u(t) &= Fx(k\tau),\quad k\tau \leq t < (k+1)\tau,~k \in \mathbb{N}_0,
	\end{align}
\end{subequations}
where 
$x(t) \in X$ is the state, $u(t) \in \mathbb{C}$ is the control input, 
$\tau>0$ is the sampling period,
$A\colon D(A) \subset X \to X$ is the generator of a $C_0$-semigroup 
$(T(t) )_{t\geq 0}$ on $X$,
$B \in \mathcal{L}(\mathbb{C},X)$ is the control operator, and
$F \in \mathcal{L}(X,\mathbb{C})$ is the feedback operator.

\begin{dfntn}
	The sampled-data system \eqref{eq:sampled_data_sys} is called
	{\em strongly stable} if 
	\[\lim_{t\to \infty} \|x(t)\| = 0
	\] 
	for every initial  state $x^0 \in X$.
\end{dfntn}

To state the main result,
we make the following assumption on the 
sampled-data system
\eqref{eq:sampled_data_sys}.
\begin{assumption}
	\label{assum:for_MR}
	Let $A$ be a Riesz-spectral operator on a 
	Hilbert space $X$ with  simple eigenvalues 
	$(\lambda_n)_{n \in \mathbb{N}}$ and
	corresponding eigenvectors $(\phi_n)_{n \in \mathbb{N}}$. 
	Let $(\psi_n)_{n \in \mathbb{N}}$ be the eigenvectors of 
	$A^*$ such that $(\phi_n)_{n \in \mathbb{N}}$ and 
	$(\psi_n)_{n \in \mathbb{N}}$ are biorthogonal.
	Let the control operator $B \in \mathcal{L}(\mathbb{C},X)$ 
	and the feedback operator $F\in \mathcal{L}(X,\mathbb{C})$ 
	be represented as
	\begin{align}
		\label{eq:B_F_rep}
		Bu = bu,\quad u \in \mathbb{C};\qquad 
		Fx = \langle x, f\rangle,\quad x\in X
	\end{align}
	for some $b,f \in X$. Assume that 
	the operators $A$, $B$, and $F$ satisfy the following conditions:
	\begin{enumerate}
		\def\theenumi{\arabic{enumi}}
		\def\labelenumi{\theenumi}
		\renewcommand{\labelenumi}{(A\arabic{enumi})}
		\item
		\label{assump:finite_unstable}
		There exist constants $\omega, \alpha, \Upsilon >0$ such that 
		$\mathbb{C}_{-\omega} \cap (\mathbb{C} \setminus \Omega_{\alpha,\Upsilon})$ has
		only finite elements of $(\lambda_n)_{n \in \mathbb{N}}$.
		\item 
		\label{assump:imaginary} 
		$\{\lambda_n :n \in \mathbb{N}\} \cap i \mathbb{R} = \emptyset$.
		\item
		\label{assump:closed_loop}
		$A+BF$ generates a polynomially stable $C_0$-semigroup 
		$(T_{BF}(t))_{t \geq0}$
		with parameter $\alpha$ on $X$.
		\item 
		\label{assump:b_f_cond}
		There exist constants
		$\beta,\gamma \geq 0$ such that
		$
		b \in  \mathcal{D}^{\beta}$,
		$f\in \mathcal{D}_*^{\gamma}
		$,
		and one of the following conditions holds:
		\begin{enumerate}
			\item $\beta ,\gamma \in \mathbb{N}_0$ and
			$\beta + \gamma \geq \alpha$.
			\item $\beta + \gamma > \alpha$.
		\end{enumerate}
		
	\end{enumerate}
\end{assumption}

\def\theenumi{\alph{enumi})}
\def\labelenumi{\theenumi}
\def\theenumii{(\roman{enumii})}
\def\labelenumii{\theenumii}

By (A\ref{assump:finite_unstable}), the Riesz-spectral operator $A$
generates a $C_0$-semigroup $(T(t))_{t \geq 0}$. 
Since $\sigma(A) = \overline{\{\lambda_n:n \in \mathbb{N}\}}$, it follows that $\sigma(A) \cap i \mathbb{R} = \emptyset$
under (A\ref{assump:finite_unstable}) and (A\ref{assump:imaginary}).
The 
eigenvalues $(\lambda_n)_{n \in \mathbb{N}}$ may approach $i \mathbb{R}$
asymptotically, and an upper bound of the asymptotic rate
is represented by the parameter $\alpha$
given in (A\ref{assump:finite_unstable}).
Note that
when $\lim_{\ell \to \infty }\re \lambda_{n_\ell} = 0$
for some subsequence $(\lambda_{n_\ell})_{\ell \in \mathbb{N}}$, 
there does not exist a feedback operator $F \in \mathcal{L}(X,\mathbb{C})$
such that the $C_0$-semigroup $(T_{BF}(t))_{t \geq0}$ generated by
$A+BF$ is exponentially stable; see, e.g., Theorem~8.2.3 of \cite{Curtain2020}.
We assume by (A\ref{assump:b_f_cond}) that 
$B$ and $F$ have stronger boundedness properties related to
the parameter $\alpha$ than the standard boundedness properties
$B \in \mathcal{L}(\mathbb{C},X)$ and $F \in \mathcal{L}(X,\mathbb{C})$.
Assumptions similar to (A\ref{assump:b_f_cond}) are placed to
perturbation operators in the robustness analysis of
polynomial stability 
developed in \cite{Paunonen2011,Paunonen2012SS,Paunonen2013SS}.
Note that not all of 
(A\ref{assump:finite_unstable})--(A\ref{assump:b_f_cond})
are imposed in every result. In fact,
(A\ref{assump:imaginary}) is not used in Section~\ref{sec:spectrum}
except for Lemma~\ref{lem:T_minus_poly_stable}, while
(A\ref{assump:closed_loop}) is not imposed in Sections~\ref{sec:Integral_RT} and \ref{sec:Integral_RSTR}.

The following theorem is the main result of this paper, which
shows that polynomial stability is robust with respect to sampling.
\begin{thrm}
	\label{thm:SD_SS}
	If Assumption~\Rref{assum:for_MR} is satisfied, then
	there exists $\tau^*>0$ such that 
	the following statements hold
	for all $\tau \in (0,\tau^*)$:
	\begin{enumerate}
		\item
		The sampled-data system \eqref{eq:sampled_data_sys} is strongly stable.
		\item
		Let $0 < \delta \leq \alpha/2$, and assume that
		$\delta \leq 1$ or $\beta \geq \alpha$. Then, 
		for every initial state $x^0 \in \mathcal{D}^{\delta}$,
		the state $x$ of the sampled-data system \eqref{eq:sampled_data_sys} 
		satisfies
		\begin{equation}
			\label{eq:solution_conv}
			\|x(t)\| 
			=
			\begin{cases}
				o(t^{-\delta/\alpha}) & \text{if $0<\delta < \alpha /2$} \vspace{5pt}\\
				o \left(\sqrt{
					\dfrac{\log t}{t} }\right) & \text{if $\delta = \alpha /2$} 
			\end{cases}
		\end{equation}
		as $t\to \infty$.
	\end{enumerate}
\end{thrm}

Let $\alpha,\delta >0$ and
consider the case $F = 0$ and $\lambda_n = -1/n^{\alpha} + in$ for $n \in \mathbb{N}$.
Then Assumption~\ref{assum:for_MR}
holds for all $B \in \mathcal{L}(\mathbb{C},X)$, where the constants $\beta$
and $\gamma$ in (A\ref{assump:b_f_cond}) are chosen
such that $\beta = 0$ and $\gamma > \alpha$.
We also have $\|T(t)(-A)^{-\delta}\| = O(t^{-\delta/\alpha})$
as $t \to \infty$, and
this decay rate is optimal in the sense that 
$\liminf_{t\to \infty} t^{\delta/\alpha}\|T(t)(-A)^{-\delta}\| >0$.
Therefore, 
one cannot replace $t^{-\delta/\alpha}$ in the estimate \eqref{eq:solution_conv} by 
functions with better decay rates.
Whether the logarithmic correction term $\sqrt{\log t}$ for the case $\delta = \alpha/2$
may be omitted remains open.

The assumption $\delta \leq 1$ implies $D(A) \subset \mathcal{D}^{\delta}$.
On the other hand, the assumption $\beta \geq \alpha$ leads to
the uniform boundedness of $\|R(z,T(\tau)) S(\tau)\|$ on an annulus
$\{z \in \mathbb{C}:
1<|z| < 1+\varepsilon\}$ with some sufficiently small $\varepsilon>0$
for a fixed $\tau >0$, where
$S(\tau) \in \mathcal{L}(\mathbb{C},X)$ is defined by
\begin{equation}
	\label{eq:S_def}
	S(\tau)u \coloneqq  \int^\tau_0 T(s) Buds,\quad u \in \mathbb{C}.
\end{equation}
We will employ these assumptions in Section~\ref{sec:Integral_RSTR}.

The proof of Theorem~\ref{thm:SD_SS} is divided into several steps.
In the next subsection, we prove the equivalence between the stability of
the sampled-data system \eqref{eq:sampled_data_sys}  and that of 
the discretized system.
Section~\ref{sec:resolvent_cond} is devoted to resolvent conditions
for the stability of discrete semigroups on Hilbert spaces.
To apply these resolvent conditions, in Section~\ref{sec:spectrum},
we investigate  the spectrum of the operator 
that represents the dynamics of the discretized system.
In Section~\ref{sec:resolvent_discretized_sys}, we complete the proof of 
Theorem~\ref{thm:SD_SS}, by using
the resolvent conditions presented in Section~\ref{sec:resolvent_cond}.

\subsection{Discretized system}
For $t\geq 0$,
define  $\Delta(t) \in \mathcal{L}(X)$ by
\begin{equation*}
	\Delta(t) \coloneqq T(t) + S(t) F.
\end{equation*}
Then the state $x$ of 
the sampled-data system \eqref{eq:sampled_data_sys}
satisfies
\begin{equation}
	\label{eq:discretized_sys}
	x\big((k+1)\tau\big) = \Delta(\tau) x(k\tau)
\end{equation}
for all $k\in \mathbb{N}_0$,
which we call the {\em discretized system}.

To prove Theorem~\ref{thm:SD_SS}, 
it suffices by the next result to investigate
the discrete semigroup
$(\Delta(\tau)^k)_{k \in \mathbb{N}}$.
\begin{prpstn}
	\label{prop:SD_DT}
	Let $A$ be the generator of a $C_0$-semigroup $(T(t))_{t\geq 0}$ on 
	a Banach space $X$. Let
	$B \in \mathcal{L}(\mathbb{C},X)$ and
	$F \in \mathcal{L}(X,\mathbb{C})$.
	The following statements hold for a fixed $\tau >0$:
	\begin{enumerate}
		\item
		The sampled-data system \eqref{eq:sampled_data_sys} is strongly stable
		if and only if the discrete semigroup $(\Delta(\tau)^k)_{k \in \mathbb{N}}$
		is strongly stable. 
		\item 
		Let $g\colon (0,\infty) \to \mathbb{R}$, and suppose that
		there exist constants $k_0 \in \mathbb{N}$ and $M_1,M_2>0$ such that 
		for all $k \in \mathbb{N}$ with $k \geq k_0$ and all $s \in [0,\tau)$,
		\begin{equation}
			\label{eq:f_bound}
			M_1 g(k\tau) \leq g(k) \leq M_2g(k\tau+s).
		\end{equation}
		Then
		the state $x$ of 
		the sampled-data system \eqref{eq:sampled_data_sys}
		with  initial state $x^0\in X$
		satisfies 
		\begin{equation}
			\label{eq:solution_conv_from_dist}
			\|x(t)\| = o \big(g(t) \big )\qquad (t \to \infty)
		\end{equation}
		if and only if $x^0$ satisfies
		\begin{align}
			\label{eq:dis_time_poly_decay}
			\|\Delta(\tau)^k x^0\| = 
			o\big(
			g(k)
			\big)\qquad (k \to \infty).
		\end{align}
	\end{enumerate}
\end{prpstn}
\begin{proof}
	The statement a) has been proved in 
	Proposition~2.2 in \cite{Wakaiki2021SIAM}, and therefore 
	we show only the statement~b).
	Assume that \eqref{eq:solution_conv_from_dist} holds for
	the state $x$ of 
	the sampled-data system \eqref{eq:sampled_data_sys}
	with initial state $x^0\in X$.
	Take $\varepsilon >0$. There exists $t_1 >0$ such that 
	for all $t \geq t_1$,
	\[
	\|x(t)\| \leq \frac{\varepsilon}{M_1} g(t).
	\]
	Choose $k_1 \in \mathbb{N}$
	so that $k_1 \geq k_0$ and $k_1\tau \geq t_1$.
	By \eqref{eq:discretized_sys} and \eqref{eq:f_bound}, 
	we have that 
	\[
	\|\Delta(\tau)^k x^0\| =
	\|x(k \tau)\| \leq 
	\frac{\varepsilon}{M_1} g(k\tau)
	\leq \varepsilon g(k)
	\]
	for all $k \geq k_1$.
	Hence, \eqref{eq:dis_time_poly_decay} holds.
	
	Conversely, assume that $x^0 \in X$ satisfies \eqref{eq:dis_time_poly_decay}, and
	take $\varepsilon >0$. We have that 
	\[
	1\leq  K \coloneqq \sup_{0\leq s < \tau} \| \Delta(s) \| < \infty.
	\]
	Then $\|x(k\tau+s)\| \leq K \|x(k\tau)\|$ 	for all $k \in \mathbb{N}_0$ and $s \in [0,\tau)$.
	By assumption,
	there exists $k_2 \in \mathbb{N}$ such that  for all $k \geq k_2$,
	\[
	\|\Delta(\tau)^k x^0\| \leq 
	\frac{\varepsilon}{KM_2} g(k).
	\]
	Combining this estimate and \eqref{eq:discretized_sys}, we obtain
	\[
	\|x(k\tau+s)\|\leq \frac{\varepsilon}{M_2} g(k)
	\]
	for all $k \geq k_2$ and $s \in [0,\tau)$.
	It follows from \eqref{eq:f_bound} that 
	\[
	\|x(k\tau+s)\| 
	\leq \varepsilon g(k\tau +s)
	\]
	for all 
	$k \geq \max \{k_0,k_2 \}$ and $s \in [0,\tau)$.
	Thus, we obtain \eqref{eq:solution_conv_from_dist}.
\end{proof}

We immediately see that 
the condition \eqref{eq:f_bound} holds for
$g(t) \coloneqq t^{-\delta}$ with $\delta >0$.
The function $g$ defined by
\[
g(t) \coloneqq \sqrt{\frac{\log t}{t}}
\]
also satisfies the condition \eqref{eq:f_bound}.
In fact, we obtain
\begin{align*}
	\frac{\log (k\tau)}{k\tau} &=
	\frac{\log (k\tau)}{\tau \log k} \cdot \frac{\log k}{k} \\
	\frac{\log k}{k} &=
	\frac{\frac{\log k}{\log (k\tau+s)}}{ \frac{k}{k\tau+s}}  \cdot
	\frac{\log (k\tau+s)}{k\tau+s} 
	\leq 
	\frac{\frac{\log k}{\log (k\tau)}}{ \frac{k}{(k+1)\tau}} \cdot
	\frac{\log (k\tau+s)}{k\tau+s}
\end{align*}
for all $k \in \mathbb{N}$ with $k > \max\{1,1/\tau \}$
and all $s \in [0,\tau)$.
Therefore, there exists $k_0 \in \mathbb{N}$ such that 
for all $k \geq k_0$ and  $s \in [0,\tau)$,
\[
\frac{\sqrt{\tau}}{2} \sqrt{ \frac{\log (k\tau)}{k\tau}}
\leq \sqrt{\frac{\log k}{k}} \leq 2 \sqrt{\tau} \sqrt{\frac{\log (k\tau+s)}{k\tau+s}}.
\]

\section{Resolvent conditions for stability of discrete semigroups}
\label{sec:resolvent_cond}
First, we review resolvent characterizations of power boundedness and 
strong stability
of discrete semigroups on Hilbert spaces.
A resolvent characterization of power bounded discrete semigroups has been
obtained in Theorem II.1.12 of \cite{Eisner2010}, which is an analogue of
the characterization of uniformly bounded $C_0$-semigroups due to
\cite{Gomilko1999,Shi2000}.
Moreover, a
resolvent characterization of strongly stable discrete semigroups has been
developed in Theorem~3.11 of 
\cite{Tomilov2001}; see also Theorem II.2.23 of \cite{Eisner2010}.
\begin{thrm}
	\label{thm:strong_stability_resol}
	Let $X$ be a Hilbert space and 
	let $\Delta \in \mathcal{L}(X)$ satisfy $\mathbb{E}_1 \subset \rho(\Delta)$.
	Then the following statements hold:
	\begin{enumerate}
		\item
		The discrete semigroup $(\Delta^k)_{k \in \mathbb{N}}$ is power bounded
		if and only if 
		\begin{align*}
			\limsup_{r \downarrow 1} \,(r-1)
			\int^{2\pi}_0 
			\|R(re^{i\theta}, \Delta)x\|^2  d\theta &< \infty\quad \text{for all $x \in X$ and} \\
			\limsup_{r \downarrow 1} \,(r-1)
			\int^{2\pi}_0 
			\|R(re^{i\theta}, \Delta)^*y\|^2  d\theta &< \infty\quad \text{for all $y \in X$.}
		\end{align*}		
		\item	
		The discrete semigroup $(\Delta^k)_{k \in \mathbb{N}}$ is strongly stable
		if and only if
		\begin{align*}
			\lim_{r \downarrow 1} \, (r-1)
			\int^{2\pi}_0 
			\|R(re^{i\theta}, \Delta)x\|^2  d\theta &=0\quad \text{for all $x \in X$ and} \\
			\limsup_{r \downarrow 1} \,(r-1)
			\int^{2\pi}_0 
			\|R(re^{i\theta}, \Delta)^*y\|^2  d\theta &< \infty\quad \text{for all $y \in X$.}
		\end{align*}
	\end{enumerate}
\end{thrm}

Next, we investigate a resolvent condition on the rate of decay for
discrete semigroups on Hilbert spaces. To this end,
the following equalities given in Lemma~II.1.11 of \cite{Eisner2010} are useful.
\begin{lmm}
	\label{lem:T_powered}
	Let $X$ be a Banach space and let $\Delta \in \mathcal{L}(X)$
	with spectral radius $\sr(\Delta)$. Then
	\[
	\Delta ^k= 
	\frac{r^{k+1}}{2\pi} \int^{2\pi}_0 e^{i\theta (k+1)} R(re^{i\theta},\Delta )d\theta = 
	\frac{r^{k+2}}{2\pi(k+1)} 
	\int^{2\pi}_0 e^{i\theta (k+1)} R(re^{i\theta},\Delta )^2d\theta
	\]
	for all $k \in \mathbb{N}$ and $r > \sr(\Delta)$.
\end{lmm}

We state the discrete analogue of Lemma~3.2 in \cite{Wakaiki2021JEE}.
\begin{prpstn}
	\label{prop:discrete_poly_decay}
	Let $(\Delta^k)_{k \in \mathbb{N}}$ be a power bounded
	discrete semigroup  on a Hilbert space $X$.
	The following statements hold for a fixed $x \in X$:
	\begin{enumerate}
		\item
		If 
		$
		\|\Delta^kx\| = o(k^{-\delta})
		$
		as $k \to \infty$ for some $0< \delta \leq 1/2$,
		then 
		\begin{align}
			\label{eq:log_cond_T}
			\lim_{r \downarrow 1}
			\Lambda_{\delta} (r)
			\int^{2\pi}_0 \|R(re^{i\theta}, \Delta) x\|^2 d\theta = 0,
		\end{align}
		where
		\[
		\Lambda_{\delta}(r) \coloneqq
		\begin{cases}
			(r-1)^{1-2\delta} & \text{if $0 < \delta < 1/2$} \\
			|\log (r-1)|^{-1} & \text{if $\delta = 1/2$.}
		\end{cases}
		\]
		\item
		If 
		\eqref{eq:log_cond_T} holds for some $0 < \delta < 1/2$, 
		then 
		$
		\|\Delta^kx\| = o(k^{-\delta})
		$
		as $k \to \infty$. On the other hand,
		if
		\eqref{eq:log_cond_T} holds for $\delta = 1/2$, 
		then 
		\begin{align*}
			\|\Delta^kx\| = o\left(\sqrt{
				\frac{\log k}{k}
			}\right)\qquad (k \to \infty).
		\end{align*}
	\end{enumerate}
\end{prpstn}
\begin{proof}
	a)
	Assume that $x \in X$ satisfies $\|\Delta^kx\| = o(
	k^{-\delta}
	)$ as $k \to \infty$ for some $0 < \delta \leq 1/2$.
	Take $\varepsilon>0$ and $1 < r < 2 $.
	There exists $k_0 \in \mathbb{N}$ such that 
	\begin{equation}
		\label{eq:Delta_small_o}
		\|\Delta^kx\|^2 \leq \frac{\varepsilon}{k ^{2\delta}}
	\end{equation}
	for all $k \geq k_0$.
	We obtain
	\begin{equation}
		\label{eq:parseval}
		\frac{1}{2\pi}
		\int^{2\pi}_0 \|R(re^{i\theta},\Delta) x\|^2 d\theta =
		\sum_{k=-\infty}^{\infty} 
		\left\|\frac{1}{2\pi}\int^{2\pi}_0 e^{ik\theta} R(re^{i\theta},\Delta) x d\theta\right\|^2
	\end{equation}
	by
	Parseval's equality for vector-valued functions, 
	which can be proved by the scalar-valued Parseval's equality
	as done for Plancherel's theorem in Section~1.8 of \cite{Arendt2001}.
	Noting that the first equality in Lemma~\ref{lem:T_powered} is true also for $k=0$,
	we have
	for each positive integer $k$,
	\[
	\frac{1}{2\pi}\int^{2\pi}_0 e^{ik\theta} R(re^{i\theta},\Delta) x d\theta
	= \frac{\Delta^{k-1}x}{r^k}.
	\]
	On the other hand, Cauchy's integral theorem implies that
	for each non-positive integer $k$,
	\[
	\int^{2\pi}_0 e^{ik\theta} R(re^{i\theta},\Delta) xd\theta
	= 
	-i r^{-k} \int_{\mathbb{T}_{1/r}} z^{-k} (I-z\Delta)^{-1}x dz = 0.
	\]
	Therefore, we have from the equality \eqref{eq:parseval}
	that
	\[
	\frac{1}{2\pi}
	\int^{2\pi}_0 \|R(re^{i\theta}, \Delta) x\|^2 d\theta
	=
	\sum_{k=0}^{\infty} \frac{\|\Delta^k x\|^2}{r^{2(k+1)}}.
	\]
	Since $(\Delta^k)_{k \in \mathbb{N}}$ is power bounded, it follows that 
	$M \coloneqq \sup_{k \in \mathbb{N}_0}\|\Delta^k\| < \infty$.
	Hence
	\begin{equation*}
		\sum_{k=0}^{k_0-1} \frac{\|\Delta^k x\|^2}{r^{2(k+1)}}
		\leq k_0 M^2 \|x\|^2.
	\end{equation*}
	
	First suppose that $0 < \delta < 1/2$. Then the estimate 
	\eqref{eq:Delta_small_o} yields
	\[
	\sum_{k=k_0}^{\infty} \frac{\|\Delta^k x\|^2}{r^{2(k+1)}} 
	\leq \frac{\varepsilon}{r^2} 
	\sum_{k=k_0}^\infty \frac{1}{k^{2\delta}r^{2k}} \leq 
	\frac{\varepsilon}{r^2} \int_0^{\infty} 
	\frac{1}{t^{2\delta} r^{2t}} dt.
	\]
	We have that 
	\begin{align*}
		\int_0^{\infty} 
		\frac{1}{t^{2\delta} r^{2t}} dt =
		\int_0^{\infty} 
		\frac{e^{-2t \log r }}{t^{2\delta}} dt 
		=
		\frac{\Gamma(1-2\delta)}{(2 \log r)^{1-2\delta}},
	\end{align*}
	where $\Gamma$ is the Gamma function.
	Since
	\[
	\frac{\log (v+1)}{v} \geq \frac{1}{2} 
	\]
	for all $0 < v < 1$, it follows that 
	\begin{equation*}
		\sum_{k=k_0}^{\infty} \frac{\|\Delta^k x\|^2}{r^{2(k+1)}}  \leq 
		\frac{\varepsilon \Gamma(1-2\delta)}{r^2(r-1)^{1-2\delta}}.
	\end{equation*}
	Hence
	\[
	\limsup_{r\downarrow 1} \, (r-1)^{1-2\delta}
	\sum_{k=0}^{\infty} \frac{\|\Delta^k x\|^2}{r^{2(k+1)}} \leq 	
	\limsup_{r\downarrow 1}\, (r-1)^{1-2\delta}
	k_0 M^2 \|x\|^2 + \limsup_{r\downarrow 1}\,\frac{\varepsilon \Gamma(1-2\delta)}{r^2}  \leq \varepsilon \Gamma(1-2\delta).
	\]
	Since $\varepsilon >0$ was arbitrary, 
	the desired conclusion \eqref{eq:log_cond_T} holds for $0 < \delta < 1/2$.
	
	Next we consider the case $\delta = 1/2$.
	By the well-known formula for the first polylogarithm (see, e.g.,
	p.~3 of \cite{Hain1994}), we obtain
	\[
	\sum_{k=1}^\infty \frac{1}{kr^{2k}} = \log \frac{r^2}{r^2-1}.
	\]
	Moreover,
	\[
	\log \frac{r^2}{r^2-1} = 
	\log 
	\frac{r}{r-1} + \log \frac{r}{r+1}
	\leq \log 
	\frac{r}{r-1} = \log r +  |\log (r-1)|.
	\]
	Therefore, the estimate \eqref{eq:Delta_small_o} gives
	\begin{equation*}
		\sum_{k=k_0}^{\infty} \frac{\|\Delta^k x\|^2}{r^{2(k+1)}} 
		\leq \frac{\varepsilon}{r^2} 
		\sum_{k=k_0}^\infty \frac{1}{kr^{2k}} \leq 
		\frac{\varepsilon}{r^2} (\log r + |\log (r-1)|).
	\end{equation*}
	Since $|\log (r-1)|^{-1} \to 0$ as $r \downarrow 1$, we have
	\[
	\limsup_{r\downarrow 1}\, |\log (r-1)|^{-1} 
	\sum_{k=0}^{\infty} \frac{\|\Delta^k x\|^2}{r^{2(k+1)}} \leq 	
	\limsup_{r\downarrow 1}\, |\log (r-1)|^{-1} 
	\left(k_0 M^2 \|x\|^2  
	+ \frac{\varepsilon \log r}{r^2}
	\right) + \limsup_{r\downarrow 1}\, \frac{\varepsilon}{r^2}  \leq \varepsilon.
	\]
	This proves that 
	\eqref{eq:log_cond_T} holds for $\delta = 1/2$.
	
	b) 	
	By Lemma~\ref{lem:T_powered} and the Cauchy-Schwartz inequality,
	we have that for all $x,y \in X$, $r > 1$, and $k \in \mathbb{N}$,
	\begin{align*}
		|\langle \Delta^k x, y 
		\rangle| 
		&\leq 
		\frac{r^{k+2}}{2\pi(k+1)}
		\int^{2\pi}_0 |
		\langle R(re^{i\theta},\Delta)^2x, y \rangle
		| d\theta \\
		&=
		\frac{r^{k+2}}{2\pi(k+1)}
		\int^{2\pi}_0 |
		\langle R(re^{i\theta},\Delta)x, R(re^{-i\theta},\Delta^*)y \rangle
		| d\theta \\
		&\leq
		\frac{r^{k+2}}{2\pi(k+1)}
		\left(
		\int^{2\pi}_0 \|R(re^{i\theta},\Delta)x \|^2 d\theta
		\right)^{1/2}
		\left(
		\int^{2\pi}_0 \|R(re^{i\theta},\Delta^*)y \|^2 d\theta
		\right)^{1/2}.
	\end{align*}
	Take $r_1>1$.
	Since $(\Delta^k)_{k \in \mathbb{N}}$ is power bounded, 
	Theorem~\ref{thm:strong_stability_resol}.a)
	and the uniform boundedness principle imply that
	there exists a constant $M_1 >0$ such that 
	\[
	(r-1)
	\int^{2\pi}_0 \|R(re^{i\theta},\Delta^*)y \|^2 d\theta 
	\leq M_1^2\|y\|^2
	\]
	for all $y \in X$ and $r \in (1,r_1)$.
	Hence
	\[
	\| \Delta^kx\| \leq 
	\frac{M_1}{2\pi} \cdot
	\frac{r^{k+2}}{(k+1)(r-1)}
	\sqrt{\frac{r-1}{\Lambda_{\delta}(r)}}
	\left(\Lambda_{\delta}(r)
	\int^{2\pi}_0 \|R(re^{i\theta},\Delta)x \|^2 d\theta
	\right)^{1/2}
	\]
	for all $x \in X$ and $r \in (1,r_1)$.
	Put $r \coloneqq 1 + \frac{1}{k+1}$. Then
	\[
	\frac{r^{k+2}}{(k+1)(r-1)} = \left(
	1+ \frac{1}{k+1}
	\right)^{k+2} \to e\qquad (k \to \infty).
	\]
	Moreover, we obtain
	\begin{align*}
		\frac{
			r-1}{\Lambda_{\delta}(r)}
		&=
		\begin{cases}
			(k+1)^{-2\delta} & \text{if $0< \delta < 1/2$} \vspace{5pt}\\
			\dfrac{\log (k+1)}{k+1} & \text{if $\delta = 1/2$}.
		\end{cases}
	\end{align*}
	Hence, there exists $k_1 \in \mathbb{N}$ such that 
	for all $k \geq k_1$,
	\[
	\frac{r^{k+2}}{(k+1)(r-1)}
	\sqrt{\frac{r-1}{\Lambda_{\delta}(r)}}
	\leq 
	\begin{cases}
		2ek^{-\delta}& \text{if $0< \delta < 1/2$} \vspace{5pt}\\
		2 e \sqrt{
			\dfrac{\log k}{k}} &  \text{if $\delta = 1/2$}.
	\end{cases}
	\]
	Combining this estimate with \eqref{eq:log_cond_T}, 
	we obtain $
	\|\Delta^kx\| = o(k^{-\delta})
	$ as $k \to \infty$ for  $0 < \delta < 1/2$ and
	\[
	\|\Delta^kx\| = o\left(\sqrt{\frac{\log k}{k}} \right)
	\] as $k \to \infty$ for $\delta = 1/2$.
\end{proof}

\section{Spectrum and sampling}
\label{sec:spectrum}
To apply Theorem~\ref{thm:strong_stability_resol}
to the 
discretized system \eqref{eq:discretized_sys},
we have to show that $\mathbb{E}_1 \subset \rho (\Delta(\tau))$
is satisfied.
The aim of this section is to prove the following theorem. 
\begin{thrm}
	\label{lem:resol_T_SF}
	If {\rm (A\ref{assump:finite_unstable})}, {\rm (A\ref{assump:closed_loop})}, 
	and {\rm (A\ref{assump:b_f_cond})} hold,
	then there exists $\tau^* >0$ 
	such that 
	$\mathbb{E}_1 \subset \rho (\Delta(\tau))$ 
	for all $\tau \in (0,\tau^*)$.
\end{thrm}

First, we apply a spectral decomposition for $A$.
Next, we prove the inclusion $\mathcal{D}^{\beta} \subset
D((-A-BF)^{\widetilde \beta})$ for all $\widetilde \beta \in [0,\beta)$.
Using this inclusion,
we also show that $|1-F(\lambda I - A)^{-1}B|$ is bounded from below
by a positive constant on $\rho(A) \cap \overline{\mathbb{C}_0}$.
This estimate for the continuous-time system
leads to an analogous estimate for the discretized system, i.e., 
a lower bound of
$|1- F(zI - T(\tau))^{-1}S(\tau)|$ on
$\rho(T(\tau)) \cap \overline{\mathbb{E}_1}$. Finally,
the desired inclusion $\mathbb{E}_1 \subset \rho (\Delta(\tau))$
is proved.

\subsection{Spectral decomposition}
\label{sec:Spec_decomp}
We start by applying a spectral decomposition for $A$ under (A\ref{assump:finite_unstable}). A more general version of
spectral decompositions for unbounded operators can
be found in Lemma~2.4.7 of \cite{Curtain2020}
and Proposition~IV.1.16 of \cite{Engel2000}.

Since only finite elements of $(\lambda_n)_{n \in \mathbb{N}}$
are in 
$\mathbb{C}_{-\omega} \cap (\mathbb{C} \setminus \Omega_{\alpha,\Upsilon})$,
there exists  a smooth, positively oriented, and simple closed curve $\Phi$ in 
$\rho(A)$ containing $\sigma(A) \cap \mathbb{C}_0$ in its interior and 
$\sigma(A) \cap (\mathbb{C} \setminus \mathbb{C}_0)$ in its exterior. 
The operator
\begin{equation}
	\label{eq:projection}
	\Pi \coloneqq \frac{1}{2\pi i} \int_{\Phi} (\lambda I - A)^{-1} d\lambda 
\end{equation}
is a projection on $X$ and yields the decomposition
\[
X = X^+ \oplus X^-,
\]
where 
\[
X^+ \coloneqq \Pi X,\quad X^- \coloneqq (I - \Pi) X.
\]
We have that $\dim X^+ < \infty$. Moreover, $X^+$ and $X^-$ are 
$T(t)$-invariant for all $t \geq 0$. Define
\begin{align*}
	A^+ \coloneqq A|_{X^+} ,\quad A^- \coloneqq A|_{X^-}.
\end{align*}
Then
\[
\sigma(A^+) = \sigma(A) \cap \mathbb{C}_0,\quad
\sigma(A^-) = \sigma(A) \cap (\mathbb{C} \setminus \mathbb{C}_0).
\]
Let $N_{\rm a},N_{\rm b}  \in \mathbb{N}$ satisfy
\begin{align}
	\label{eq:Ns_def}
	\{\lambda_n:1\leq n \leq N_{\rm a} - 1 \}  &= \{\lambda_n 
	: n \in \mathbb{N}\} \cap \mathbb{C}_0
	= \sigma(A^+) \\
	\label{eq:lambda_cond_large}
	\{\lambda_n:1\leq n \leq N_{\rm b} - 1 \}  &= \{\lambda_n 
	: n \in \mathbb{N}\} \cap  \big(\mathbb{C}_{-\omega} \cap (\mathbb{C} \setminus \Omega_{\alpha,\Upsilon})\big)
\end{align}
by changing the order of $(\lambda_n)_{n \in \mathbb{N}}$ if necessary.
By construction, we obtain $N_{\rm b} \geq  N_{\rm a}$.
The series expansions of $A^+$ and $A^-$ are given by
\begin{align*}
	A^+x^+ &= \sum_{n=1}^{N_{\rm a}-1} 
	\lambda_n \langle x^+ , \psi_n \rangle \phi_n 
	\quad \text{for all~} x ^+ \in D(A^+) = X^+ \\
	A^-x^- &= \sum_{n=N_{\rm a}}^\infty \lambda_n \langle x^- , \psi_n \rangle \phi_n\quad \text{for all~} x^- \in D(A^-) =
	\left\{
	x^- \in X^-: \sum_{n=N_{\rm a}}^\infty |\lambda_n|^2\, |\langle x^-, \psi_n \rangle|^2 < \infty
	\right\}.
\end{align*}

For all $\lambda \in \rho(A)$, $X^+$ and $X^-$ are 
$(\lambda I - A)^{-1}$-invariant and
\[
(\lambda I - A^+)^{-1} =
(\lambda I - A)^{-1}|_{X^+},\quad (\lambda I - A^-)^{-1} =
(\lambda I - A)^{-1}|_{X^-}.
\]
For $t \geq 0$, we define
\[
T^+(t) \coloneqq T(t)|_{X^+},\quad T^-(t) \coloneqq T(t)|_{X^-}.
\]
Then $(T^+(t))_{t\geq 0}$ 
and $(T^-(t))_{t\geq 0}$ are $C_0$-semigroups with generators $A^+$ and $A^-$,
respectively.
The adjoint $\Pi^*$ 
is also a projection on $X$ and 
yields a spectral decomposition for $A^*$.
We define
\[
X^+_* \coloneqq \Pi^*X,\quad 
X^-_* \coloneqq (I-\Pi^*)X.
\]
The restriction $A^-_*  \coloneqq A^*|_{X^-_*}$ of the adjoint $A^*$
is the generator of a $C_0$-semigroup 
$(T^-_*(t))_{t\geq 0}$, where
\[
T^-_*(t)\coloneqq T(t)^*|_{X^-_*}
\]
for $t \geq 0$. 
Define 
\begin{align*}
	B^+ \coloneqq \Pi B,\quad B^- \coloneqq (I-\Pi)B,\quad 
	F^+ \coloneqq F|_{X^+},\quad F^- \coloneqq F|_{X^-}.
\end{align*}
From \eqref{eq:B_F_rep}, we have that
\begin{align*}
	B^+u = b^+ u,\quad B^- u = b^-u\quad &\text{for all~} u \in \mathbb{C} \\ 
	F^+x^+ = \langle x^+, f^+\rangle\quad &\text{for all~} x^+ \in X^+ \\
	F^-x^- = \langle x^-, f^-\rangle \quad &\text{for all~} x^- \in X^-, 
\end{align*}
where 
$b^+ \coloneqq \Pi b$, $b^- \coloneqq (I-\Pi)b$, 
$f^+ \coloneqq \Pi^*f$, and $f^- \coloneqq (I-\Pi^*)f$.

The polynomial stability of
$(T^{-}(t))_{t\geq 0}$
and $(T^{-}_*(t))_{t\geq 0}$ under {\rm (A\ref{assump:finite_unstable})} and 
{\rm (A\ref{assump:imaginary})}
is an immediate consequence of 
the equivalence between a) and c) in Theorem~\ref{thm:decay_charac}.
\begin{lmm}
	\label{lem:T_minus_poly_stable}
	If {\rm (A\ref{assump:finite_unstable})} and 
	{\rm (A\ref{assump:imaginary})} hold, then
	the $C_0$-semigroups $(T^{-}(t))_{t\geq 0}$
	and $(T^{-}_*(t))_{t\geq 0}$
	constructed as above
	are polynomially stable with parameter $\alpha$.
\end{lmm}

\subsection{Inclusion $\mathcal{D}^{\beta} \subset 
		D((-A-BF)^{\widetilde \beta})$ for $0 \leq \widetilde \beta < \beta $}
Let a Riesz-spectral operator  $A$ on a Hilbert space $X$ 
generate a $C_0$-semigroup. Let $B \in \mathcal{L}(\mathbb{C},X)$ and
$F \in \mathcal{L}(X,\mathbb{C})$
be such that $A+BF$
is the generator of  a uniformly bounded $C_0$-semigroup.
Then,
the fractional power $(-A-BF)^{\beta}$ is well defined for every 
$\beta >0$.
We will show that 
if 
$\ran (B) \subset \mathcal{D}^{\beta}$ for some $\beta >0$, then
$\mathcal{D}^{\beta} \subset D((-A-BF)^{\widetilde \beta})$ holds for all
$\widetilde \beta \in [0,\beta)$.
To this end, the following result
is useful; see Lemma~5.4 of \cite{Wakaiki2021JEE} for the proof.
\begin{lmm}	
	\label{lem:A_AD_domain}
	Let $X$ be a Banach space and let $V \in \mathcal{L}(X)$. 
	Suppose that $A$ and $A+V$ are the generators of exponentially stable 
	$C_0$-semigroups 
	on $X$. 
	Then $D((-A)^{\alpha_1}) \subset D((-A-V)^{\alpha_2})$
	for all $\alpha_1,\alpha_2 \in (0,1)$ with $\alpha_2 < \alpha_1$.
\end{lmm}

We investigate the relation between $D(A^n)$ and $D((A+BF)^n)$
for $n \in \mathbb{N}$.
\begin{lmm}
	\label{lem:integer_case}
	Let $A$ be a linear operator on a Banach space $X$ and let 
	$F \in \mathcal{L}(X,\mathbb{C})$.
	Define
	$B\in \mathcal{L}(\mathbb{C},X)$ by
	$Bu \coloneqq bu$ for $u \in \mathbb{C}$, where $b \in X$.
	Then the following assertion holds for all $n \in \mathbb{N}$:
	If $x \in D(A^n)$ and $b \in D(A^{n-1})$, then
	$x \in D((A+BF)^n)$ and 
	\begin{equation}
		\label{eq:A+BF_powered}
		(A+BF)^n x = A^n x+q_{n-1} A^{n-1}b + \cdots + q_0 b,
	\end{equation}
	where $q_{m} \coloneqq F(A+BF)^{n-m-1}x \in \mathbb{C}$ for  $m=0,\dots,n-1$.
\end{lmm}	

\begin{proof}
	We prove the assertion by induction.
	In the case $n=1$, 
	$x \in D(A)$ satisfies $x \in D(A+BF)$ and $(A+BF)x = Ax + (Fx)b$.
	Now,
	assume that the assertion holds for some $n \in \mathbb{N}$.
	Let $x \in D(A^{n+1})$ and $b \in D(A^{n})$.
	Then $A^{n} x \in D(A)$ and
	\[
	A^{m} b \in D(A)
	\]
	for all $m = 0,\dots,n-1$.
	This and the inductive assumption imply
	\[
	(A+BF)^nx \in D(A) = D(A+BF),
	\] 
	and hence
	$x \in D((A+BF)^{n+1})$.
	Moreover,
	\begin{align*}
		(A+BF)^{n+1}x &= 
		A (A^n x+ (Fx) A^{n-1}b + \cdots + (F(A+BF)^{n-1} x) b ) +
		BF(A+BF)^{n}x\\
		&=
		A^{n+1} x + (Fx) A^{n}b + \cdots + (F(A+BF)^{n-1} x) Ab + (F(A+BF)^{n}x)b.  
	\end{align*}
	Thus, \eqref{eq:A+BF_powered} holds when $n$ is replaced by $n+1$.
\end{proof}

Combining Lemmas~\ref{lem:A_AD_domain} and \ref{lem:integer_case},
we obtain the following result.
\begin{lmm}
	\label{lem:ABF_domain}
	Let  $A$ be a Riesz-spectral operator on a Hilbert space $X$ 
	with simple eigenvalues $(\lambda_n)_{n \in \mathbb{N}}$
	such that 
	$\sup_{n \in \mathbb{N}} \re \lambda_n < \infty$ and 
	$0$ is not an accumulation point of the set $\{ \lambda_n : n \in \mathbb{N}  \}$.
	Define
	$B\in \mathcal{L}(\mathbb{C},X)$ by
	$Bu \coloneqq bu$ for $u \in \mathbb{C}$, where  $b \in \mathcal{D}^{\beta}$ for some $\beta >0$.
	Let $F \in \mathcal{L}(X,\mathbb{C})$ be such that
	$A+BF$ generates a uniformly bounded $C_0$-semigroup on $X$.
	Then for all $\widetilde \beta \in   [0,\beta)$,  one has
	$\mathcal{D}^{\beta} \subset D((-A-BF)^{\widetilde \beta})$;
	in particular $b \in D((-A-BF)^{\widetilde \beta})$.
\end{lmm}

\begin{proof}
	There exists $h >0$ such that $A_h \coloneqq A-h I$ generates 
	an exponentially stable $C_0$-semigroup on $X$. 
	To prove that  $
	\mathcal{D}^{\beta} = D((-A_h)^{\beta})
	$, 
	we take $r >0$ so that $\mathbb{D}_r$ has a finite number of the eigenvalues 
	$(\lambda_n)_{n \in \mathbb{N}}$. Let $N_{r}  \in \mathbb{N}$ satisfy
	\[
	\{\lambda_n:1\leq n \leq N_{r} - 1 \}  = \{\lambda_n 
	: n \in \mathbb{N}\} \cap \mathbb{D}_r
	\]
	by changing the order of $(\lambda_n)_{n \in \mathbb{N}}$ if necessary.
	We decompose $x \in X$ into
	\[
	x = \sum_{n=1}^{N_r-1} 
	\langle x , \psi_n \rangle \phi_n  + 
	\sum_{n=N_r}^{\infty} 
	\langle x , \psi_n \rangle \phi_n \eqqcolon
	x_1+ x_2.
	\]
	By construction,
	$x_1 \in \mathcal{D}^{\beta} \cap D((-A_h)^{\beta})$.
	Since $\inf_{n \geq N_r} |\lambda_n| \geq r >0$, the equivalence
	between 
	$x_2 \in \mathcal{D}^{\beta}$ and
	$x_2 \in D((-A_h)^{\beta})$ follows as
	in the proof of Lemma~3.2.11.c of \cite{Curtain2020}.
	Therefore, we obtain $
	\mathcal{D}^{\beta} = D((-A_h)^{\beta})
	$.
	
	Since 
	\[
	D\big((-A_h)^{\beta} \big) \subset D\Big((-A_h)^{\widetilde \beta} \Big),\quad 
	D\big((-A-BF)^{\beta}\big) \subset D\Big((-A-BF)^{\widetilde \beta}\Big)
	\] 
	for every $\widetilde \beta \in [0,\beta)$,
	it suffices to consider the case where $n < \widetilde \beta < \beta < n+1$
	for some $n \in \mathbb{N}_0$.
	Put $\beta_0 \coloneqq \beta - n$ and 
	$\widetilde \beta_0 \coloneqq \widetilde \beta - n$.	
	Take $x\in \mathcal{D}^{\beta} = D((-A_h)^{\beta}) $.
	We have from the first law of exponents (see, e.g., Proposition~3.1.1.c) of \cite{Haase2006}) that 
	\begin{align}
		D\big((-A_h)^\beta\big) &= \left\{
		x \in D(A_h^n ) : A_h^nx \in D\big((-A_h)^{\beta_0}\big) 
		\right\} \label{eq:A_beta_domain}\\
		D\Big((-A_h-BF)^{\widetilde \beta}\Big) &= \left\{
		x \in D\big((A_h+BF)^n\big) : (A_h+BF)^nx \in D\Big((-A_h-BF)^{\widetilde \beta_0}\Big) 
		\right\}.\label{eq:ABF_beta_domain}
	\end{align}
	Since $x,b \in D(A_h^{n})$, 
	Lemma~\ref{lem:integer_case} implies  that
	$x \in D((A_h+BF)^{n})$ and
	\begin{equation}
		\label{eq:Aeps_BF}
		(A_h+BF)^{n}x = A_h^nx + q_{n-1} A_h^{n-1}b + \cdots + q_0 b
	\end{equation}
	for some $q_0,\dots,q_{n-1} \in \mathbb{C}$.
	By $b\in D(A_h^{n}) $, 
	\[
	A_h^{m}b \in D(A_h) \subset   D\big((-A_h)^{\beta_0}\big) 
	\]
	for all $m = 0,\dots,n-1$.
	Moreover, we have from $x \in D((-A_h)^\beta)$ 
	and \eqref{eq:A_beta_domain} that 
	\[
	A_h^nx \in D\big((-A_h)^{\beta_0}\big).
	\]
	Hence $(A_h+BF)^{n}x \in D((-A_h)^{\beta_0}) $
	by \eqref{eq:Aeps_BF}.
	Lemma~\ref{lem:A_AD_domain} yields
	\[
	D\big((-A_h)^{\beta_0}\big) \subset D\Big((-A_h-BF)^{\widetilde \beta_0}\Big),
	\]
	and therefore \[
	(A_h+BF)^{n}x \in D\Big((-A_h-BF)^{\widetilde \beta_0}\Big).
	\]
	This and 
	\eqref{eq:ABF_beta_domain} give
	$x \in D((-A_h-BF)^{\widetilde \beta})$. Since 
	$D((-A_h-BF)^{\widetilde \beta}) = D((-A-BF)^{\widetilde \beta})$
	by Proposition~3.1.9.a) of \cite{Haase2006},
	we conclude that $\mathcal{D}^{\beta} \subset 
	D((-A-BF)^{\widetilde \beta})$.
\end{proof}

\subsection{Lower bound of $|1 - FR(z, T(\tau)) S(\tau)|$}
In this subsection, we complete the proof of Theorem~\ref{lem:resol_T_SF},
by showing  that 
$|1 - FR(z,T(\tau))S(\tau)|$ is bounded from below
by a positive constant on $\rho(T(\tau)) \cap \overline{\mathbb{E}_1}$.
First, we estimate $|F R(\lambda, A+BF)B|$
with the help of Lemma~\ref{lem:ABF_domain}.
\begin{lmm}
	\label{lem:FRB_bound}
	If {\rm (A\ref{assump:finite_unstable})}, {\rm (A\ref{assump:closed_loop})}, 
	and {\rm (A\ref{assump:b_f_cond})} hold, then
	there exist constants $M \geq 1$, $\widetilde \beta \in [0,\beta]$, and $
	\widetilde \gamma \in [0,\gamma]$ such that 
	\begin{equation}
		\label{eq:bf_inclusions}
		b \in D\Big((-A-BF)^{\widetilde \beta }\Big),\quad 
		f \in D\Big((-A^*-F^*B^*)^{\widetilde \gamma}\Big)
	\end{equation}
	and
	\begin{equation}
		\label{eq:F_RABF_B_bound}
		|F R(\lambda, A+BF)B| \leq M \big\|
		(-A-BF)^{\widetilde \beta} b \big\|\,\big\|
		(-A^*-F^*B^*)^{\widetilde \gamma} f \big\|
	\end{equation}
	for all $\lambda \in 
	\overline{\mathbb{C}_0}$.
\end{lmm}

\begin{proof}
	If $\beta,\gamma \in \mathbb{N}_0$, then
	we have from Lemma~\ref{lem:integer_case} that 
	$b \in D((A+BF)^\beta)$ and $f \in D((A^*+F^*B^*)^\gamma)$.
	Therefore, \eqref{eq:bf_inclusions} holds with
	$\widetilde \beta = \beta$ and $\widetilde \gamma = \gamma$.
	If $\beta+\gamma > \alpha$, then
	Lemma~\ref{lem:ABF_domain} implies that \eqref{eq:bf_inclusions} 
	and $\widetilde \beta +  \widetilde \gamma \geq \alpha$ hold
	for some $\widetilde \beta \in [0,\beta)$ and $
	\widetilde \gamma \in [0,\gamma)$.
	For such $\widetilde \beta$ and $\widetilde \gamma$,
	the inequality \eqref{eq:F_RABF_B_bound} immediately follows from Lemma~\ref{lem:resolvent_estimate}.
\end{proof}

Using
Lemma~\ref{lem:FRB_bound}, we next obtain  an estimate of 
$|1 - FR(\lambda,A)B|$.
\begin{lmm}
	\label{lem:cont_time_trans_func_bound}
	If {\rm (A\ref{assump:finite_unstable})}, {\rm (A\ref{assump:closed_loop})}, 
	and {\rm (A\ref{assump:b_f_cond})} hold, then
	there exists $\varepsilon>0$ such that 
	\[
	|1 - FR(\lambda,A)B| > \varepsilon
	\]
	for all $\lambda \in \rho(A) \cap \overline{\mathbb{C}_0}$.
\end{lmm}
\begin{proof}
	Let $\lambda \in \rho (A)$. Then
	\[
	\lambda I - A - BF = (\lambda I - A)(I - (\lambda I - A)^{-1}BF).
	\]
	Since $\sigma((\lambda I - A)^{-1}BF) \setminus \{0 \} 
	= \sigma(F(\lambda I - A)^{-1}B) \setminus \{0 \} $ (see, e.g., 
	(3) in Section~III.2 of \cite{Gohberg1990}),
	we obtain
	\[
	\lambda \in  \rho (A+BF) \quad \Leftrightarrow \quad 
	1 \in \rho \big((\lambda I - A)^{-1}BF\big) \quad \Leftrightarrow \quad 
	1 \in \rho (F(\lambda I - A)^{-1}B).
	\]
	Moreover, a simple calculation shows that
	\begin{equation}
		\label{eq:extension}
		\frac{1}{1 - F(\lambda I - A)^{-1}B} = F(\lambda I - A - BF)^{-1}B + 1
	\end{equation}
	for all $\lambda \in \rho(A) \cap \rho(A+BF)$.
	Hence Lemma~\ref{lem:FRB_bound} implies that 
	there exists  $\varepsilon>0$ such that 
	$|1 - FR(\lambda ,A)B| > \varepsilon$ for 
	all $\lambda \in \rho(A) \cap \overline{\mathbb{C}_0} $.
\end{proof}

The estimate on the continuous-time system obtained in 
Lemma~\ref{lem:cont_time_trans_func_bound} leads to
an analogous estimate on the discretized system
as in the robustness analysis of exponential stability \cite{Rebarber2006} and 
strong stability \cite{Wakaiki2021SIAM}.
To show this, we use 
the series expansion of $R(z,T(\tau)) S(\tau)$ under  (A\ref{assump:finite_unstable}),
where $S(\tau)$ is defined by \eqref{eq:S_def}.
If $0 \in \rho(A)$, then $S(\tau)$ 
is written as
\[
S(\tau) = A^{-1} (T(\tau) - I) B = 
\sum_{n=1}^{\infty} \frac{e^{\tau \lambda_n} -1}{\lambda_n}
\langle b, \psi_n \rangle  \phi_n,
\]
and hence the series expansion of $R(z,T(\tau)) S(\tau)$ is given by
\begin{equation}
	\label{eq:RS_series}
	R(z,T(\tau)) S(\tau) = \sum_{n=1}^{\infty} \frac{e^{\tau \lambda_n}-1}{z - e^{\tau \lambda_n}} \cdot
	\frac{\langle b, \psi_n \rangle }{\lambda_n} \phi_n
\end{equation}
for $z \in \rho(T(\tau))$.
If $0 \not\in \rho(A)$, then $0$ is a simple eigenvalue of $A$ under (A\ref{assump:finite_unstable}). Let 
$n_0 \in \mathbb{N}$ satisfy
$\lambda_{n_0} = 0$. Analogously, we obtain
\begin{equation}
	\label{eq:RS_series_zero_case}
	R(z,T(\tau)) S(\tau) = 
	\frac{\tau}{z - 1} 
	\langle b, \psi_{\lambda_{n_0}} \rangle \phi_{\lambda_{n_0}}
	+\sum_{n\not=n_0} \frac{e^{\tau \lambda_n}-1}{z - e^{\tau \lambda_n}} \cdot
	\frac{\langle b, \psi_n \rangle }{\lambda_n} \phi_n
\end{equation}
for $z \in \rho(T(\tau))$.

Recall that $N_{\rm b} \in \mathbb{N}$ is chosen so that 
\eqref{eq:lambda_cond_large} holds.
Therefore, $\lambda_n \not=0$ for all $n \geq N_{\rm b}$.
For 
each $n \geq N_{\rm b}$, the
$n$th term of the  series expansion of $R(z,T(\tau)) S(\tau)$
satisfies the following estimate, which 
is obtained from  arguments similar to those in the proofs of Theorem~2.1 in \cite{Rebarber2006} 
and Lemma~3.8 in \cite{Wakaiki2021SIAM}.
\begin{lmm}
	\label{lem:frac_lam_bound}
	Suppose that {\rm (A\ref{assump:finite_unstable})}  holds. Let 
	$\widetilde \alpha \geq \alpha$ and let
	$N_{\rm b} \in \mathbb{N}$ be such that \eqref{eq:lambda_cond_large}
	holds.
	Then there exist constants $\Upsilon_1,\Upsilon_2 >0$ such that 
	\begin{equation}
		\label{eq:case1-3_bound}
		\left|
		\frac{1 - e^{\tau \lambda_n}}{z - e^{\tau \lambda_n}}
		\right| \, \left|
		\frac{1}{\lambda_n}
		\right| \leq \max\big\{\Upsilon_1, 
		\Upsilon_2 |\lambda_n|^{\widetilde \alpha}
		\big\}
	\end{equation}
	for all $\tau >0$, $z \in \overline{\mathbb{E}_1}$, and
	$n \geq N_{\rm b}$.
\end{lmm}
\begin{proof}
	Take $\tau > 0$ and $z \in \overline{\mathbb{E}_1}$.
	Let $N_{\rm b} \in \mathbb{N}$ be as in \eqref{eq:lambda_cond_large}.
	For $n \geq N_{\rm b}$, we divide the proof into three
	cases: (i)~$\tau \re \lambda_n \leq -1$; (ii)~$\tau \re \lambda_n > -1$
	and $\re\lambda_n \leq -\omega$; and (iii)~$\tau \re \lambda_n > -1$
	and $\lambda_n \in \Omega_{\alpha,\Upsilon}$. 
	For all cases, the following inequality is useful:
	\begin{equation}
		\label{eq:dis_time_bound}
		\left|
		\frac{1 - e^{\tau \lambda_n}}{z - e^{\tau \lambda_n}}
		\right|
		\leq \frac{|1 - e^{\tau \lambda_n}|}{1 - e^{\tau \re \lambda_n}} = 
		\frac{\frac{|1 - e^{\tau \lambda_n}|}{\tau  |\lambda_n|}}{ \frac{1 - e^{\tau \re \lambda_n}}{\tau |\re \lambda_n|}} \cdot
		\frac{|\lambda_n|}{|\re \lambda_n|}
	\end{equation}
	for all $n \geq N_{\rm b}$.

	First we consider the case~(i) $\tau \re \lambda_n \leq -1$.
	By \eqref{eq:lambda_cond_large},
	we obtain
	$|\lambda_n| \geq \kappa$
	for all $n \geq N_{\rm b}$ and some 
	$\kappa >0$. 
	Therefore, the estimate \eqref{eq:dis_time_bound} gives
	\begin{equation}
		\label{eq:case1_bound}
		\left|
		\frac{1 - e^{\tau \lambda_n}}{z - e^{\tau \lambda_n}}
		\right| \, \left|
		\frac{1}{\lambda_n}
		\right| \leq 
		\frac{|1 - e^{\tau \lambda_n}|}{1 - e^{\tau \re \lambda_n}} \left|
		\frac{1}{\lambda_n}
		\right| \leq \frac{2}{(1 - e^{-1}) \kappa}.
	\end{equation}
	Next we examine the case~(ii) $\tau \re \lambda_n > -1$
	and $\re\lambda_n \leq -\omega$.
	The function
	\[
	g(\lambda) \coloneqq
	\begin{cases}
		\dfrac{1 - e^{\lambda}}{\lambda} & \text{if $\lambda \not=0$} \vspace{5pt}\\
		-1 & \text{if $\lambda = 0$}
	\end{cases}
	\]
	is holomorphic on $\mathbb{C}$. 
	Therefore, 
	there exists $M_1 >0$ such that $|g(\lambda)| \leq M_1$
	for all $\lambda \in \mathbb{C}$ satisfying $-1 \leq \re \lambda \leq 0$
	and $|\im \lambda| \leq \pi$.
	For all 
	$\lambda \in \mathbb{C}$ with  $|\im \lambda | \leq \pi$ and all $\ell \in \mathbb{N}$, 
	we obtain
	\[
	|g(\lambda \pm 2\ell \pi i)| = \left|\frac{1 - e^{\lambda}}{\re \lambda + i(\im \lambda \pm 2\ell \pi )}\right|
	\leq |g(\lambda)|.
	\]
	Hence $|g(\lambda)| \leq M_1$ if $-1 \leq \re \lambda \leq 0$.
	This estimate on $g$ shows that
	\begin{equation}
		\label{eq:dis_bount1}
		\frac{|1 - e^{\tau \lambda_n}|}{\tau |\lambda_n|} \leq M_1.
	\end{equation}
	Moreover, 
	applying 
	the mean value theorem to the function $t \mapsto e^t$ on $[-1,0]$,
	we obtain
	\[
	\frac{1 - e^t}{|t|} \geq e^{-1}
	\]
	for all $t \in [-1,0]$. This and the substitution $t = \tau \re \lambda_n$ imply
	\begin{equation}
		\label{eq:dis_bount2}
		\frac{1 - e^{\tau \re \lambda_n}}{\tau |\re \lambda_n|} \geq e^{-1}.
	\end{equation}
	From
	the estimates \eqref{eq:dis_time_bound},
	\eqref{eq:dis_bount1}, and \eqref{eq:dis_bount2},
	we have
	\begin{equation}
		\label{eq:case2_bound}
		\left|
		\frac{1 - e^{\tau \lambda_n}}{z - e^{\tau \lambda_n}}
		\right| \, \left|
		\frac{1}{\lambda_n}
		\right| \leq \frac{e M_1 }{|\re \lambda_n |} \leq \frac{e M_1 }{\omega}.
	\end{equation}
	
	Finally, we study the case~(iii)
	$\tau \re \lambda_n > -1$
	and $\lambda_n \in \Omega_{\alpha,\Upsilon}$.
	It follow from $\lambda_n \in \Omega_{\alpha,\Upsilon}$ that
	\[
	|\re \lambda_n| \geq  \frac{\Upsilon}{|\im \lambda_n|^\alpha}.
	\]
	Using the fact that $|\lambda_n| \geq \kappa >0$ 
	for all $n \geq N_{\rm b}$, we have that for $\widetilde \alpha \geq \alpha$,
	\begin{align}
		\label{eq:real_part_estimate}
		\frac{1}{|\re \lambda_n|} \leq
		\frac{|\im \lambda_n|^\alpha}{\Upsilon} \cdot 
		\frac{	|\lambda_n|^{\widetilde \alpha}}
		{|\lambda_n|^{\widetilde \alpha}} \leq 
		\frac{|\lambda_n|^{\widetilde \alpha}}
		{\Upsilon |\lambda_n|^{\widetilde \alpha - \alpha}} \leq 
		\frac{	|\lambda_n|^{\widetilde \alpha} }{\Upsilon \kappa^{\widetilde \alpha-\alpha}
		}.
	\end{align}
	The estimates \eqref{eq:dis_bount1} and \eqref{eq:dis_bount2}
	hold also in the case (iii).
	Applying the estimates \eqref{eq:dis_bount1},
	\eqref{eq:dis_bount2}, and \eqref{eq:real_part_estimate} to \eqref{eq:dis_time_bound} yields
	\begin{align}
		\label{eq:case3_bound}
		\left|
		\frac{1 - e^{\tau \lambda_n}}{z - e^{\tau \lambda_n}}
		\right| \,\left|
		\frac{1}{\lambda_n}
		\right| \leq \frac{e M_1}{\Upsilon \kappa^{\widetilde \alpha-\alpha}}	|\lambda_n|^{\widetilde \alpha} .
	\end{align}

	Define 
	the constants $\Upsilon_1,\Upsilon_2 > 0$ by
	\[
	\Upsilon_1 \coloneqq 
	\max\left\{
	\frac{2}{(1 - e^{-1}) \kappa},~\frac{e M_1 }{\omega}
	\right\},\quad 
	\Upsilon_2 \coloneqq 
	\frac{e M_1}{\Upsilon \kappa^{\widetilde \alpha-\alpha}},
	\]
	which are independent of 
	$\tau > 0$ and $z \in \overline{\mathbb{E}_1}$.
	From  the estimates \eqref{eq:case1_bound}, \eqref{eq:case2_bound}, 
	and \eqref{eq:case3_bound},
	we conclude
	that \eqref{eq:case1-3_bound} holds 
	for all $N \geq N_{\rm b}$.
\end{proof}

\begin{lmm}
	\label{lem:discrete_time_trans_func_bound}
	Suppose that {\rm (A\ref{assump:finite_unstable})}  holds.
	Let $b \in \mathcal{D}^{\beta}$ and $f  \in \mathcal{D}_*^{\gamma}$
	for some $\beta,\gamma \geq 0$ satisfying 
	$\beta + \gamma \geq \alpha$.
	If there exists $\varepsilon_{\rm c}\in (0,1)$ such that 
	\begin{equation}
		\label{eq:continuous_lower_bound}
		|1 -  FR(\lambda, A)B| > \varepsilon_{\rm c} 
	\end{equation}
	for all $\lambda \in \rho(A) \cap \overline{\mathbb{C}_0}$,
	then, for any $\varepsilon_{\rm d}  \in (0,\varepsilon_{\rm c})$,
	there exists $\tau^*>0$ such that
	\begin{equation}
		\label{eq:discrete_lower_bound}
		\big|1 - FR\big(z, T(\tau)\big) S(\tau)\big| > \varepsilon_{\rm d}
	\end{equation}
	for all $\tau \in (0,\tau^*)$ and $z \in \rho(T(\tau)) \cap \overline{\mathbb{E}_1}$.
\end{lmm}

The proof of 
Lemma~\ref{lem:discrete_time_trans_func_bound}
is based on the approximation
approach developed in the proof of 
Theorem~2.1 of \cite{Rebarber2006} for
the preservation of exponential stability 
under sampling.
We decompose
the transfer functions ${\mathbf G}(\lambda) \coloneqq F(\lambda I - A)^{-1}B$ and
${\mathbf H}_{\tau}(z) \coloneqq F(zI - T(\tau))^{-1}S(\tau)$ into
finite-dimensional truncations and 
infinite-dimensional tails with approximation order $N \in \mathbb{N}$:
\begin{align*}
	{\mathbf G}(\lambda) &= 
	\sum_{n = 1}^{N-1} \frac{\langle b , \psi_n \rangle \langle \phi_n, f\rangle}{\lambda - \lambda_n}  + 
	\sum_{n = N}^\infty \frac{\langle b , \psi_n \rangle \langle \phi_n, f\rangle}{\lambda - \lambda_n},\quad  \lambda \in \rho(A)\\
	{\mathbf H}_{\tau}(z) &=\sum_{n = 1}^{N-1}
	\frac{e^{\tau \lambda_n}-1}{z - e^{\tau \lambda_n}} \cdot 
	\frac{\langle b , \psi_n \rangle 
		\langle\phi_n,f \rangle}{\lambda_n} + 
	\sum_{n = N}^{\infty}
	\frac{e^{\tau \lambda_n}-1}{z - e^{\tau \lambda_n}} \cdot 
	\frac{\langle b , \psi_n \rangle 
		\langle\phi_n,f \rangle}{\lambda_n},\quad  z \in \rho\big(T(\tau)\big),
\end{align*}
where,  for simplicity of notation, 
we assume 
that $0 \not\in \{\lambda_n:n \in \mathbb{N}\}$.
The main idea of 
the approximation
approach in \cite{Rebarber2006} is twofold.
First, we prove that 
the infinite-dimensional tails become arbitrarily small
as $N$ increases. Next, we show that
if $\tau>0$ is sufficiently small, then
the finite-dimensional truncations with a fixed $N \in \mathbb{N}$
are close (except near the unstable poles) under the  relationship
$z = e^{\tau \lambda}$ of the variable $\lambda$ in the continuous-time
setting and 
the variable $z$ in the discrete-time setting.
For the infinite-dimensional tails, a treatment different from the previous studies \cite{Rebarber2006,Wakaiki2021SIAM} is
required due to the geometric property of the eigenvalues of 
the generator $A$ and the conditions on 
the control operator $B$
and the feedback operator $F$. On the other hand,
the analysis of 
the finite-dimensional truncations has no difficulty arising
from polynomial stability. Hence, to the finite-dimensional truncations, 
one can apply 
the arguments  developed in the proof of 
Theorem~2.1 of \cite{Rebarber2006} with only minor modifications; see also 
the proof of Lemma~3.8 of \cite{Wakaiki2021SIAM}.

\begin{proof}[Proof of Lemma~\Rref{lem:discrete_time_trans_func_bound}]
	{\em Step 1:}
	Let $N_{\rm b} \in \mathbb{N}$ be such that 
	\eqref{eq:lambda_cond_large} holds.
	We show that 
	for all $\varepsilon >0$, there exists 
	$N_0^{\rm c} \geq N_{\rm b}$ such that 
	\begin{equation}
		\label{eq:continuous_large_case}
		\sup_{\lambda \in \overline{\mathbb{C}_0} }
		\left|
		\sum_{n = N}^\infty \frac{\langle b , \psi_n \rangle \langle \phi_n, f\rangle}{\lambda - \lambda_n} 
		\right| \leq \varepsilon
	\end{equation}
	for all $N\geq N_0^{\rm c}$.

	As in the spectral decomposition described in 
	Section~\ref{sec:Spec_decomp},
	there exists  a smooth, positively oriented, and simple closed 
	curve $\Phi_{\rm b}$ in 
	$\rho(A)$ containing $\{\lambda_n:1\leq n \leq N_{\rm b} - 1 \} $ in its interior and 
	$\sigma(A) \setminus \{\lambda_n:1\leq n \leq N_{\rm b} - 1 \} $ in its exterior.
	Define the projection $\Pi_{\rm b}$ on $X$ by
	\[
	\Pi_{\rm b} \coloneqq \frac{1}{2\pi i} \int_{\Phi_{\rm b}} (\lambda I - A)^{-1} d\lambda,
	\]
	and put $X_{\rm b}^- \coloneqq (I - \Pi_{\rm b}) X$.
	For $t \geq 0$, define  
	$
	T_{\rm b}^-(t) \coloneqq T(t)|_{X_{\rm b}^-}.
	$
	As in  Lemma~\ref{lem:T_minus_poly_stable}, 
	$(T_{\rm b}^-(t))_{t\geq 0}$ 
	is a polynomially stable $C_0$-semigroup with parameter $\alpha$ 
	on $X_{\rm b}^-$.
	We denote by $A^-_{\rm b}$ the generator of $(T_{\rm b}^-(t))_{t\geq 0}$.
	
	Theorem~\ref{thm:decay_charac}
	implies that
	\[
	M\coloneqq 
	\sup_{\lambda \in \overline{\mathbb{C}_0}}
	\|R(\lambda , A^-_{\rm b}) (-A^-_{\rm b})^{-\alpha}\| <\infty.
	\]
	For all  $n \geq N_{\rm b}$ and $\lambda \in \overline{\mathbb{C}_0}$,
	\begin{align*}
		\frac{M_{\rm a} }{|\lambda-\lambda_n|^2\,  |\lambda_n|^{2\alpha}}
		&\leq \|R(\lambda,A^-_{\rm b})(-A^-_{\rm b})^{-\alpha} \phi_n \|^2\\
		&\leq \|R(\lambda,A^-_{\rm b})(-A^-_{\rm b})^{-\alpha}\|^2 \, \| \phi_n \|^2 \\
		&\leq M^2M_{\rm b}.
	\end{align*}
	Therefore, 
	\[
	\sup_{\lambda \in \overline{\mathbb{C}_0}}
	\frac{1}{|\lambda-\lambda_n|\,  |\lambda_n|^{\alpha}} \leq 
	M\sqrt{\frac{M_{\rm b} }{M_{\rm a}}}
	\]
	for all $n \geq N_{\rm b}$.
	By \eqref{eq:lambda_cond_large}, there exists a constant $\kappa >0$ such that  $|\lambda_n| \geq \kappa$ 
	for all $n \geq N_{\rm b}$.
	The Cauchy-Schwartz inequality implies that 
	for all $N \geq N_{\rm b}$ and $\lambda \in \overline{\mathbb{C}_0}$,
	\begin{align*}
		\left|
		\sum_{n= N}^\infty
		\frac{\langle b, \psi_n\rangle  \langle \phi_n, f \rangle }
		{\lambda - \lambda_n}
		\right| &\leq
		\sum_{n= N}^\infty
		\frac{1}{|\lambda-\lambda_n|\,  |\lambda_n|^{\alpha}} \cdot 
		\frac{|\lambda_n|^{\beta+\gamma}  \, 
			|\langle b, \psi_n\rangle \langle \phi_n,f \rangle |	
		}{|\lambda_n|^{\beta+\gamma-\alpha} } \\
		&\leq 
		\frac{M}{\kappa^{\beta+\gamma-\alpha}}  \sqrt{\frac{M_{\rm b} }{M_{\rm a}}
			\left(\sum_{n= N}^\infty |\lambda_n|^{2\beta} \, |\langle b, \psi_n\rangle|^2\right)
			\left(\sum_{n= N}^\infty |\lambda_n|^{2\gamma} \,  |\langle \phi_n, f \rangle|^2\right)}.
	\end{align*}
	Since $b \in \mathcal{D}^{\beta}$ and $f  \in \mathcal{D}_*^{\gamma}$,
	we obtain
	\begin{equation}
		\label{eq:b_bounded}
		\sum_{n= 1}^\infty |\lambda_n|^{2\beta}\,  |\langle b, \psi_n\rangle|^2 < \infty,\qquad
		\sum_{n= 1}^\infty |\lambda_n|^{2\gamma} \, |\langle \phi_n, f \rangle|^2 < \infty.
	\end{equation}
	Hence, 	for all $\varepsilon >0$, there exists 
	$N_0^{\rm c} \geq N_{\rm b}$ such that \eqref{eq:continuous_large_case} holds.
	
	{\em Step 2:}	
	We shall  show that 
	for all $\varepsilon >0$, there exists $N_0^{\rm d} \geq N_{\rm b}$
	such that 
	\begin{equation}
		\label{eq:discrete_large_case}
		\sup_{z \in  \overline{\mathbb{E}_1}} 
		\left|
		\sum_{n = N}^\infty 
		\frac{1 - e^{\tau \lambda_n}}{z - e^{\tau \lambda_n}} \cdot 
		\frac{\langle b , \psi_n \rangle 
			\langle\phi_n,f \rangle}{\lambda_n}
		\right| \leq  \varepsilon
	\end{equation}
	for all $\tau >0$ and $N \geq N_0^{\rm d}$.
	Note that $N_0^{\rm d}$ is independent of $\tau$.
	
	By Lemma~\ref{lem:frac_lam_bound}
	with $\widetilde \alpha \coloneqq \beta+\gamma$, 
	there are constants $\Upsilon_1,\Upsilon_2 >0$
	such that 
	\[
	\left|
	\frac{1 - e^{\tau \lambda_n}}{z - e^{\tau \lambda_n}} \cdot 
	\frac{\langle b , \psi_n \rangle 
		\langle\phi_n,f \rangle}{\lambda_n}
	\right| \leq \left(\Upsilon_1 + 
	\Upsilon_2 |\lambda_n|^{\beta+\gamma}
	\right) |\langle b , \psi_n \rangle| \,
	|\langle\phi_n,f \rangle|.
	\]
	for all $\tau > 0$, $z \in \overline{\mathbb{E}_1}$, and $n \geq N_{\rm b}$.
	Using the Cauchy-Schwartz inequality, we obtain
	\begin{align}
		\sum_{n = N}^\infty 
		\left|
		\frac{1 - e^{\tau \lambda_n}}{z - e^{\tau \lambda_n}} \cdot 
		\frac{\langle b , \psi_n \rangle 
			\langle\phi_n,f \rangle}{\lambda_n}
		\right| &\leq 
		\Upsilon_1 \sum_{n = N}^\infty   |\langle b , \psi_n \rangle  | \,
		|\langle\phi_n,f \rangle | + 
		\Upsilon_2 \sum_{n = N}^\infty 	|\lambda_n|^{\beta+\gamma } \, |
		\langle b , \psi_n \rangle  | \,
		|\langle\phi_n,f \rangle | \notag \\
		&\leq 
		\Upsilon_1 \sqrt{\left(
			\sum_{n = N}^\infty   |\langle b , \psi_n \rangle  |^2\right) \left( \sum_{n = N}^\infty 
			|\langle\phi_n,f \rangle |^2\right)} \notag \\&\quad + \Upsilon_2
		\sqrt{
			\left(\sum_{n= N}^\infty |\lambda_n|^{2\beta} \, |\langle b, \psi_n\rangle|^2\right)
			\left(\sum_{n= N}^\infty |\lambda_n|^{2\gamma} \,  |\langle \phi_n, f \rangle|^2\right)} \label{eq:discrete_time_transfer_suff_large}
	\end{align}
	for all $N \geq N_{\rm b}$.
	As in Step 1,
	it follows 
	from \eqref{eq:b_bounded} and 
	\[
	\sum_{n= 1}^\infty  |\langle b, \psi_n\rangle|^2 < \infty,\quad
	\sum_{n= 1}^\infty  |\langle \phi_n, f \rangle|^2 < \infty
	\]
	that for all $\varepsilon >0$,
	there exists $N_0^{\rm d}  \geq N_{\rm b}$ such that 
	\eqref{eq:discrete_large_case} holds.
	
	{\em Step 3:}	
	Let 
	$\varepsilon_{\rm c} \in (0,1)$ satisfy \eqref{eq:continuous_lower_bound}, and
	choose $\varepsilon \in (0,\varepsilon_{\rm c}/3)$ arbitrarily.
	We have shown in Steps~1 and 2 that
	there exists $N_0 \geq N_{\rm b}$ such that 
	for all $N \geq N_0$ and $\tau >0$,
	\begin{subequations}
		\begin{align}
			\sup_{\lambda \in \overline{\mathbb{C}_0} } 
			&\left|
			\sum_{n=N}^{\infty} \frac{\langle b , \psi_n \rangle \langle \phi_n, f \rangle}{\lambda - \lambda_n} 
			\right| \leq \varepsilon \quad \text{and} \label{eq:suff_large_cont}\\
			\sup_{z \in  \overline{\mathbb{E}_1}} 
			&\left|
			\sum^{\infty}_{n=N}\frac{1-e^{\tau \lambda_n}}{z - e^{\tau \lambda_n}}
			\cdot  \frac{\langle b , \psi_n \rangle \langle \phi_n, f \rangle}{\lambda_n}
			\right| \leq \varepsilon. \label{eq:suff_large_dist}
		\end{align}
	\end{subequations}

	Let $N \geq N_0$.
	For simplicity of notation, we assume that 
	$\lambda_n$ is non-zero
	for all $1 \leq n \leq N-1$. 
	When $\lambda_n = 0$ for some $1 \leq n \leq N-1$,
	the corresponding term,
	\[
	\frac{e^{\tau \lambda_n}-1}{z - e^{\tau \lambda_n}} \cdot 
	\frac{\langle b , \psi_n \rangle 
		\langle\phi_n,f \rangle}{\lambda_n},
	\]
	is just replaced
	by
	\[
	\frac{\tau \langle b , \psi_n \rangle 
		\langle\phi_n,f \rangle}{z-1}
	\]
	as in \eqref{eq:RS_series_zero_case}.
	We  investigate the finite-dimensional truncation
	\[
	\sum_{n = 1}^{N-1}
	\frac{e^{\tau \lambda_n}-1}{z - e^{\tau \lambda_n}} \cdot 
	\frac{\langle b , \psi_n \rangle 
		\langle\phi_n,f \rangle}{\lambda_n}.
	\]
	This finite sum
	has no difficulty arising from polynomial stability, and hence
	we can apply the result on exponential stability developed in
	\cite{Rebarber2006}, which is outlined for  completeness.

	For $\tau, \eta, a>0$, define   
	the sets $\Omega_0$, $\Omega_1$, $\Omega_2$, and
	$\Omega_3$ by
	\begin{align*}
		\Omega_0 &\coloneqq 
		\{z = e^{\tau \lambda }: \re \lambda \geq 0,~|\tau \lambda | < \eta\} 
		=
		\{z = e^{\mu }: \re \mu \geq 0,~|\mu| < \eta\} \\
		\Omega_1 &\coloneqq 
		\{z = e^{\tau \lambda }: |\lambda - \lambda_n| \geq a ~\text{for all 
			$ 1\leq n \leq N-1$}\} \\
		&\qquad \cup
		\{z = e^{\tau \lambda }: 0< |\lambda - \lambda_n| < a\text{~and~}
		\langle b , \psi_n \rangle \langle \phi_n, f \rangle = 0
		~\text{for some $ 1\leq n \leq N-1$} \} \\
		\Omega_2 &\coloneqq \{z = e^{\tau \lambda }: 0< |\lambda - \lambda_n| < a
		\text{~and~}
		\langle b , \psi_n \rangle \langle \phi_n, f \rangle \not= 0
		~\text{for some $ 1\leq n \leq N-1$} \} \\
		\Omega_3 &\coloneqq \overline{\mathbb{E}_1} \setminus 
		\Omega_0.
	\end{align*}
	Take $0< \eta < \pi$. Then, for each $z \in \Omega_0$,
	there uniquely exists $\lambda \in\overline{\mathbb{C}_0}$ such that  
	$z = e^{\tau \lambda}$  and $|\tau \lambda| < \eta$.
	This $\lambda$ is the complex variable in the continuous-time setting
	corresponding to the complex variable $z$ in the discrete-time setting.
	Put
	$
	a^* \coloneqq \min \{|\lambda_n - \lambda_m|/2: 1\leq n,m \leq N-1\}.
	$
	Then there is no $\lambda \in \mathbb{C}$ such that one has
	both
	$|\lambda - \lambda_n| < a^*$ and 
	$|\lambda - \lambda_m| < a^*$ 
	for some $1\leq n ,m \leq N-1$ with $n \not=m$.
	By Steps 3) and 4) of the proof of Theorem~2.1 in \cite{Rebarber2006},
	there exist
	$\tau^* >0$, $\eta \in (0,\pi)$, and 
	$a \in (0,a^*)$ 
	such that
	the following three statements hold
	for all $\tau \in (0,\tau^*)$:
	\begin{enumerate}
		\item \label{it:Omega_3} 
		For 
		all $1 \leq n \leq N-1$,
		one has $e^{\tau \lambda_n} \in \mathbb{C} \setminus \Omega_3$.
		\item \label{it:Omega_4} 
		For all $z \in \Omega_0 \cap \Omega_1 \eqqcolon \Omega_4$ and
		the corresponding $\lambda \in\overline{\mathbb{C}_0}$ satisfying
		$z = e^{\tau \lambda}$  and $|\tau \lambda| < \eta$,
		\begin{equation}
			\label{eq:finite_case1}
			\left|
			\sum_{n=1}^{N-1} \frac{\langle b , \psi_n \rangle \langle \phi_n, f \rangle}{\lambda - \lambda_n}  +
			\sum^{N-1}_{n=1}\frac{1-e^{\tau \lambda_n}}{z - e^{\tau \lambda_n}}
			\cdot  \frac{\langle b , \psi_n \rangle \langle \phi_n, f \rangle}{\lambda_n} 
			\right| < \varepsilon.
		\end{equation}
		\item \label{it:Omega_5} 
		For all  $z \in (\Omega_0 \cap \Omega_2) \cup
		\Omega_3 \eqqcolon \Omega_5$,
		\begin{equation}
			\label{eq:finite_case2}
			\left|
			1 + \sum^{N-1}_{n=1}\frac{1-e^{\tau \lambda_n}}{z - e^{\tau \lambda_n}}
			\cdot  \frac{\langle b , \psi_n \rangle \langle \phi_n, f \rangle}{\lambda_n} 
			\right|> \varepsilon_{\rm c}.
		\end{equation}
	\end{enumerate}
	In what follows, $\tau,\eta, a>0$ are chosen so that
	the above statements \ref{it:Omega_3}--\ref{it:Omega_5} hold.
	
	Suppose that $z \in \rho(
	T(\tau) ) \cap \Omega_4$, and let
	$\lambda \in \overline{\mathbb{C}_0}$ satisfy
	$z = e^{\tau \lambda}$  and $|\tau \lambda| < \eta$.
	Then $\lambda \in \rho(A)$.
	Combining the estimates
	\eqref{eq:suff_large_cont}, \eqref{eq:suff_large_dist}, and
	\eqref{eq:finite_case1} with the assumption~\eqref{eq:continuous_lower_bound}, i.e.,
	\[
	|1 -  F(\lambda I - A)^{-1}B | =
	\left|
	1 - \sum_{n=1}^{\infty} \frac{\langle b , \psi_n \rangle \langle \phi_n, f \rangle}{\lambda - \lambda_n} 
	\right| > \varepsilon_{\rm c},
	\]
	we obtain
	\begin{align*}
		\left| 
		1 + \sum_{n=1}^\infty \frac{1-e^{\tau \lambda_n}}{z - e^{\tau \lambda_n}}
		\cdot  \frac{\langle b , \psi_n \rangle \langle \phi_n, f \rangle}{\lambda_n} 
		\right| &> 
		\varepsilon_{\rm c} - 3\varepsilon.
	\end{align*}
	On the other hand, if $z \in \rho(
	T(\tau) ) \cap\Omega_5$, then \eqref{eq:suff_large_dist}
	and 
	\eqref{eq:finite_case2} yield
	\begin{align*}
		\left| 
		1 + \sum_{n=1}^\infty \frac{1-e^{\tau \lambda_n}}{z - e^{\tau \lambda_n}}
		\cdot  \frac{\langle b , \psi_n \rangle \langle \phi_n, f \rangle}{\lambda_n} 
		\right| 
		&> \varepsilon_{\rm c}-\varepsilon.
	\end{align*}
	
	{\em Step 4:}
	It remains to show that
	\begin{equation}
		\label{eq:Omega45}
		\big(
		\rho\big(
		T(\tau) \big) \cap\Omega_4\big) \cup \big(\rho\big(
		T(\tau) \big) \cap\Omega_5 \big)
		= 
		\rho\big (T(\tau)\big) \cap \overline{\mathbb{E}_1}.
	\end{equation}	
	By definition,
	\[
	(\Omega_0 \cap \Omega_1)  \cup 
	(\Omega_0 \cap \Omega_2) = 
	\Omega_0 \cap (\Omega_1 \cup \Omega_2) = 
	\Omega_0 \setminus \{e^{\tau \lambda_n}:  1\leq n \leq N-1  \}.
	\]
	Moreover, the statement \ref{it:Omega_3} above shows that 
	\[
	\Omega_3 \cap 
	\{
	e^{\tau \lambda_n} : 1 \leq n \leq N-1
	\} = \emptyset.
	\]
	Hence
	\begin{align*}
		\Omega_4 \cup \Omega_5 
		&=
		(\Omega_0 \setminus \{e^{\tau \lambda_n}:  1\leq n \leq N - 1  \} )
		\cup \Omega_3 \\
		&= (
		\Omega_0 \cup \Omega_3
		) 
		\setminus 
		\{e^{\tau\lambda_n} : 1 \leq n \leq N-1 \} 
		\\
		&=
		\overline{\mathbb{E}_1}  \setminus 
		\{e^{\tau\lambda_n} : 1 \leq n \leq N-1  \}.
	\end{align*}
	This yields
	\begin{align*}
		\big(
		\rho\big(
		T(\tau) \big) \cap\Omega_4\big) \cup \big(\rho\big(
		T(\tau) \big) \cap\Omega_5 \big) &=
		\rho\big(
		T(\tau) \big) \cap (\Omega_4 \cup \Omega_5) \\
		&=
		\rho\big(
		T(\tau) \big) \cap \big(\overline{\mathbb{E}_1}  \setminus 
		\{e^{\tau\lambda_n} : 1 \leq n \leq N-1 \}\big).
	\end{align*}
	Since
	$\sigma(T(\tau)) = \overline{\{e^{\tau \lambda_n}:
		n \in \mathbb{N} \}}$,
	we have that
	\begin{align*}
		\rho\big(
		T(\tau) \big) \cap \big(\overline{\mathbb{E}_1}  \setminus 
		\{e^{\tau\lambda_n} : 1 \leq n \leq N-1 \}\big) 
		&=	\rho\big(
		T(\tau) \big) \cap \overline{\mathbb{E}_1}.
	\end{align*}
	Thus, \eqref{eq:Omega45} holds. 
\end{proof}

The following result can be obtained by a slight modification of
the proof of Lemma~4.6 in \cite{Wakaiki2021SIAM}.
\begin{lmm}
	\label{lem:circle_resol}
	Let $A$ be a Riesz-spectral operator on a Hilbert space $X$
	with simple eigenvalues 
	$(\lambda_n)_{n \in \mathbb{N}}$.
	Let $B \in \mathcal{L}(\mathbb{C},X)$ and 
	$F \in \mathcal{L}(X,\mathbb{C})$ be such that 
	$A+BF$ generates a uniformly bounded $C_0$-semigroup on $X$.
	Suppose that only finite elements of $(\lambda_n)_{n \in \mathbb{N}}$
	are contained in $\mathbb{C}_0$.
	If $\tau >0$ satisfies
	\begin{enumerate}
		\item
		$\tau (\lambda_n - \lambda_m) \not= 2\ell \pi i$
		for all $\ell \in \mathbb{Z} \setminus \{0\}$ 
		and $n,m \in \mathbb{N}$ 
		with $\lambda_n,\lambda_m \in \mathbb{C}_{0}$; and
		\item
		$
		F R (z, T(\tau))S(\tau) \not=1$
		for all $z \in \rho (T(\tau) ) \cap\mathbb{E}_1$,
	\end{enumerate}
	then  $\mathbb{E}_1 \subset \rho(\Delta(\tau))$.
\end{lmm}

The desired inclusion  $\mathbb{E}_1 \subset \rho (\Delta(\tau))$ 
follows from Lemmas~\ref{lem:cont_time_trans_func_bound}, \ref{lem:discrete_time_trans_func_bound}, and \ref{lem:circle_resol}.
\begin{proof}[Proof of Theorem~\Rref{lem:resol_T_SF}:]
	Lemmas~\ref{lem:cont_time_trans_func_bound} and \ref{lem:discrete_time_trans_func_bound},
	together with (A\ref{assump:finite_unstable}), show that
	there exist $\varepsilon >0$ and $\tau^* >0$ such that 
	for all $\tau \in (0,\tau^*)$,
	\begin{enumerate}
		\item
		$\tau (\lambda_n - \lambda_m) \not= 2\ell \pi i$
		for all $\ell \in \mathbb{Z} \setminus \{0\}$ and 
		$n,m \in \mathbb{N}$
		with $1 \leq n,m \leq N_{\rm a} - 1$; and
		\item
		$|1 - F R(z, T(\tau))S(\tau)| 
		>\varepsilon$ for all 
		$z \in \rho (T(\tau)) \cap \overline{\mathbb{E}_1}$.
	\end{enumerate}
	Hence we obtain $\mathbb{E}_1 \subset \rho(\Delta(\tau))$
	for all $\tau \in (0,\tau^*)$
	by  Lemma~\ref{lem:circle_resol}.
\end{proof}

\section{Application of resolvent conditions to discretized system}
\label{sec:resolvent_discretized_sys}
In this section, we complete the proof of the main result, Theorem~\ref{thm:SD_SS}.
To do so, we prove that 
for a sufficiently small sampling period $\tau>0$,
the operator $\Delta(\tau) = T(\tau) + S(\tau)F$ satisfies
the integral conditions on resolvents given in
Theorem~\ref{thm:strong_stability_resol} and 
Proposition~\ref{prop:discrete_poly_decay}.
We divide the resolvent $R(z,\Delta(\tau))$
into two terms, by applying
the well-known Sherman-Morrison-Woodbury formula
presented in the next lemma. 
This formula can be obtained from a straightforward calculation.
\begin{lmm}
	\label{lem:SMW}
	Let $X$ and $U$ be Banach spaces and let 
	$A\colon D(A)\subset X \to X$ be a closed linear operator. Take
	$B \in \mathcal{L}(U,X)$, $F \in \mathcal{L}(X,U)$, and $\lambda \in \rho(A)$.
	If $1 \in \rho (FR(\lambda,A)B) $, then $\lambda \in \rho (A+BF)$ and 
	\[
	R(\lambda,A+BF) = R(\lambda,A) + R(\lambda,A) B 
	(I-FR(\lambda,A)B)^{-1} FR(\lambda,A).
	\]
\end{lmm}

Suppose that Assumption~\ref{assum:for_MR} hold.
By 
Lemmas~\ref{lem:cont_time_trans_func_bound} and \ref{lem:discrete_time_trans_func_bound},
if the sampling period $\tau>0$ is sufficiently small, then 
for all $z \in \rho (T(\tau)) \cap \overline{\mathbb{E}_1}$,
one has
$1 \in \rho (FR(z,T(\tau))S(\tau)) $.
Hence the Sherman-Morrison-Woodbury formula presented in 	Lemma~\ref{lem:SMW} yields
\[
R(z, T(\tau)+S(\tau)F)  =
R\big(z, T(\tau)\big) +
\frac{R\big(z, T(\tau)\big) S(\tau)FR\big(z, T(\tau)\big) }
{1 - FR\big(z, T(\tau)\big) S(\tau) }.
\]
In what follows, we separately investigate 
the integrals of
$
\|R(z, T(\tau))\|^2$ and 
$\|R(z, T(\tau)) S(\tau)FR(z, T(\tau))\|^2$.

\subsection{Integral of $\|R(z, T(\tau))\|^2$}
\label{sec:Integral_RT}
We obtain the following result on the integral of 
$\|R(z, T(\tau))\|^2$ on circles in $\mathbb{C}$.
\begin{lmm}
	\label{lem:Resol_T_int}
	If  {\rm (A\ref{assump:finite_unstable})} and {\rm (A\ref{assump:imaginary})} 
	hold, then
	the $C_0$-semigroup $(T(t))_{t\geq 0}$ satisfies
	the following properties for a fixed $\tau >0$:
	\begin{enumerate}
		\item One has
		\begin{subequations}
			\label{eq:Resol_T_int}
			\begin{align}
				\lim_{r\downarrow 1} ~(r-1)
				\int^{2\pi}_0 
				\big\|R\big(re^{i\theta}, T(\tau)\big) x\big\|^2 d\theta &= 0\quad
				\text{for all $x \in X$ and} \label{eq:Resol_T_inta}\\
				\lim_{r\downarrow 1} ~(r-1)
				\int^{2\pi}_0 
				\big\|R\big(re^{i\theta}, T(\tau)^*\big) y\big\|^2 d\theta &= 0\quad
				\text{for all $y \in X$}.\label{eq:Resol_T_intb}
			\end{align}
		\end{subequations}
		\item 
		Let $0 < \delta \leq \alpha/2$. Then 
		\begin{subequations}
			\label{eq:poly_Resol_T_int}
			\begin{align}
				\lim_{r\downarrow 1} \Lambda_{\delta/\alpha}(r)
				\int^{2\pi}_0 
				\big\|R\big(re^{i\theta}, T(\tau)\big) x\big\|^2 d\theta &= 0 \quad
				\text{for all $x \in \mathcal{D}^{\delta}$ and} \label{eq:poly_Resol_T_inta}\\
				\lim_{r\downarrow 1} \Lambda_{\delta/\alpha}(r)
				\int^{2\pi}_0 
				\big\|R\big(re^{i\theta}, T(\tau)^*\big) y\big\|^2 d\theta &= 0\quad
				\text{for all $y \in \mathcal{D}^{\delta}_*$},\label{eq:poly_Resol_T_intb}
			\end{align}
			where $\Lambda_{\delta/\alpha}$ is defined as in Proposition~\Rref{prop:discrete_poly_decay}.
		\end{subequations}
	\end{enumerate}
\end{lmm}
\begin{proof}
	a) 
	Take $\tau >0$.
	To obtain \eqref{eq:Resol_T_inta}, we
	apply the spectral decomposition by the projection $\Pi$ given in \eqref{eq:projection}. 
	Let $x \in X$, and 
	define $x^+\coloneqq \Pi x \in X^+$ and $x^-\coloneqq (I-\Pi)x \in X^-$.
	Since $(d_1+d_2)^2 \leq 2(d_1^2+d_2^2)$ 
	for every $d_1,d_2 \geq 0$,
	it follows that in order to show \eqref{eq:Resol_T_inta}, it suffices 
	to show that 
	\[
	\lim_{r \downarrow 1} ~(r-1)
	\int^{2\pi}_0 	\big\|R\big(re^{i\theta}, T(\tau)\big) x^+\big\|^2 d\theta = 0
	\quad \text{and} \quad
	\lim_{r \downarrow 1} ~(r-1)
	\int^{2\pi}_0 	\big\|R\big(re^{i\theta}, T(\tau)\big) x^-\big\|^2 d\theta = 0.
	\]
	
	There exist constants $r_0 > 1$ and
	$c_0 >0$ such that $|r e^{i \theta} - e^{\tau \lambda_n} | \geq c_0$ for all $r \in (1,r_0)$,
	$\theta \in [0,2\pi)$, and
	$1 \leq n \leq  N_{\rm a}-1$, where $N_{\rm a} \in \mathbb{N}$ satisfies
	\eqref{eq:Ns_def}. We have that 
	for all $r \in (1,r_0)$,
	\begin{align}
		\int^{2\pi}_0 
		\big\|R\big(r e^{i \theta}, T(\tau)\big) x^+\big\|^2 d\theta
		&\leq M_{\rm b} \sum_{n=1}^{N_{\rm a}-1} |\langle 
		x^+ , \psi_n
		\rangle|^2 \int^{2\pi}_0 \frac{1}{|re^{i\theta} - e^{\tau \lambda_n} |^2 } d\theta \notag \\
		&\leq \frac{2\pi M_{\rm b}}{ c_0^2} \sum_{n=1}^{N_{\rm a}-1} |\langle 
		x^+ , \psi_n
		\rangle|^2.
		\label{eq:x_plus_bound}
	\end{align}
	Therefore,
	\[
	\lim_{r\downarrow 1} ~(r-1)
	\int^{2\pi}_0 
	\big\|R\big(r e^{i \theta}, T(\tau)\big) x^+\big\|^2  d\theta = 0.
	\]
	Since 
	the discrete semigroup
	$(T^-(\tau)^k )_{k \in \mathbb{N}}$ is strongly stable
	by Lemma~\ref{lem:T_minus_poly_stable},  we see from
	Theorem~\ref{thm:strong_stability_resol} that
	\[
	\lim_{r\downarrow 1} ~(r-1)
	\int^{2\pi}_0 
	\big\|R\big(r e^{i \theta}, T^-(\tau)\big) x^-\big\|^2  d\theta =0.
	\]
	Hence \eqref{eq:Resol_T_inta} holds.
	Applying the spectral decomposition for $A^*$ as in the case of $A$,
	we obtain \eqref{eq:Resol_T_intb}.
	
	b) Take
	$x \in \mathcal{D}^{\delta}$, and
	define $x^+\coloneqq \Pi x \in X^+$ and $x^-\coloneqq (I-\Pi)x \in X^-$.
	From the estimate \eqref{eq:x_plus_bound}, we obtain
	\[
	\lim_{r\downarrow 1} \Lambda_{\delta/\alpha}(r)
	\int^{2\pi}_0 
	\big\|R\big(r e^{i \theta}, T(\tau)\big) x^+\big\|^2  d\theta = 0.
	\]
	By Lemma~\ref{lem:T_minus_poly_stable}, 
	$(T^-(t))_{t \geq 0}$ is polynomially stable with parameter $\alpha$.
	Since  $x^- \in D((-A^-)^{\delta})$, it follows
	from Theorem~\ref{thm:decay_charac} that 
	\[
	\|T^-(\tau)^k x^-\| = 
	\|T^-(k\tau) x^-\| =  o\left(
	\frac{1}{k^{\delta/\alpha}}
	\right) \qquad (k \to \infty).
	\]
	Proposition~\ref{prop:discrete_poly_decay}.a) implies that
	\[
	\lim_{r\downarrow 1} \Lambda_{\delta/\alpha}(r)
	\int^{2\pi}_0 
	\big\|R\big(r e^{i \theta}, T^-(\tau)\big) x^-\big\|^2  d\theta =0.
	\]
	Therefore, \eqref{eq:poly_Resol_T_inta} holds. Analogously, 
	we obtain
	\eqref{eq:poly_Resol_T_intb} by the spectral decomposition for $A^*$.
\end{proof}

\subsection{Integral of $\|R(z, T(\tau)) S(\tau)FR(z, T(\tau))\|^2$}
\label{sec:Integral_RSTR}
Next, we study
the integral of $\|R(z, T(\tau)) S(\tau) FR(z, T(\tau))\|^2$ on circles in $\mathbb{C}$.
\begin{prpstn}
	\label{prop:RSFR_bound}
	Suppose that {\rm (A\ref{assump:finite_unstable})} and {\rm (A\ref{assump:imaginary})} hold.
	Let $b \in \mathcal{D}^{\beta}$
	and $f \in \mathcal{D}_*^{\gamma}$
	for some $\beta ,\gamma \geq 0$ satisfying $\beta+ \gamma \geq \alpha$.
	Then the following statements hold for a fixed $\tau >0$:
	\begin{enumerate}
		\item 
		One has
		\begin{align}
			\label{eq:RSFR_estimate}
			\lim_{r \downarrow 1} ~(r-1)
			\int^{2\pi}_0\big\|R\big(re^{i\theta}, T(\tau)\big) S(\tau)\big\|^2 \,  \big\|FR\big(re^{i\theta}, T(\tau)\big)\big\|^2 d\theta  = 0.
		\end{align}
		\item Let $0< \delta \leq \alpha/2$. If $\delta \leq 1$ or if $\beta \geq \alpha$, then
		\begin{subequations}
			\label{eq:RSFR_estimate_poly_decay}
			\begin{align}
				&\lim_{r\downarrow 1}
				\Lambda_{\delta/\alpha}(r)
				\int^{2\pi}_0 \big\|R\big(re^{i\theta}, T(\tau)\big) S(\tau)\big\|^2 
				\,
				\big\|F^+R\big(re^{i\theta}, T^+(\tau)\big) \big\|^2 d\theta = 0 \quad \text{and}
				\label{eq:RSFR_estimate_poly_decay_p}\\
				&\lim_{r\downarrow 1}
				\Lambda_{\delta/\alpha}(r)
				\int^{2\pi}_0 \big\|R\big(re^{i\theta}, T(\tau)\big) S(\tau)\big\|^2 
				\,
				\big\|F^-R\big(re^{i\theta}, 
				T^-(\tau)\big) (-A^-)^{-\delta} \big\|^2 d\theta= 0,
				\label{eq:RSFR_estimate_poly_decay_m}
			\end{align}
		\end{subequations}
		where $\Lambda_{\delta/\alpha}$ is defined as in Proposition~\Rref{prop:discrete_poly_decay}.
	\end{enumerate}
\end{prpstn}

To prove Proposition~\ref{prop:RSFR_bound},
we start with a simple result.
Recall that 
$b^+$, $b^-$, $f^+$, and $f^-$ were defined as
$b^+ \coloneqq \Pi b$, $b^- \coloneqq (I-\Pi) b$, $f^+ \coloneqq 
\Pi^* f$, and $f^- \coloneqq (I - \Pi^*)  f$ in Section~\ref{sec:Spec_decomp}.
\begin{lmm}
	\label{lem:F_adjoint}
	Suppose that
	{\rm (A\ref{assump:finite_unstable})} holds. 
	Let $\tau >0$ and $z \in \rho(T(\tau))$.
	Under the spectral decomposition described in Section~\ref{sec:Spec_decomp},
	the following inequalities hold for a fixed $\delta >0$:
	\begin{enumerate}
		\item
		$\|F^+R(z,T^+(\tau))\| 
		\leq  \|R(\overline{z},T(\tau)^*) f^+\|$. \vspace{3pt}
		\item
		$\|F^-R(z,T^-(\tau)) (-A^-)^{-\delta}\| 
		\leq
		\|
		R(\overline{z}, 
		T(\tau)^*) (-A^-_*)^{-\delta}f^-
		\|.
		$
	\end{enumerate}
\end{lmm}
\begin{proof}
	Let $\tau >0$ and $z \in \rho(T(\tau))$ be given.
	The inequality in a) follows from
	\begin{align*}
		\big\|F^+R\big(z,T^+(\tau)\big)\big\| 
		&= \sup\big\{ \big|F^+R\big(z,T^+(\tau)\big) x^+ \big|:x^+ \in X^+ \text{~with~} \|x^+\| = 1 \big\} \\
		&=\sup\big\{ \big| \big\langle 
		R\big(z,T(\tau)\big) x^+,f^+ \big\rangle \big|
		:x^+ \in X^+ \text{~with~} \|x^+\| = 1 \big\} \\
		&=\sup\big\{ \big|\big\langle 
		x^+,R\big(\overline{z},T(\tau)^*\big) f^+ \big\rangle \big|
		:x^+ \in X^+ \text{~with~} \|x^+\| = 1 \big\} \\
		&\leq \big \|R\big(\overline{z},T(\tau)^*\big) f^+ \big\|.
	\end{align*}
	In order to obtain the inequality in b), we again use
	\begin{align*}
		\big\|F^-R\big(z,T^-(\tau)\big) (-A^-)^{-\delta}\big\| 
		&=\sup\big\{ \big|\big\langle
		R\big(z,T(\tau)\big) (-A^-)^{-\delta}x^- ,f^-\big\rangle \big|:
		x^- \in X^- \text{~with~} \|x^-\| = 1 \big\}.
	\end{align*}
	A routine calculation shows that 
	\begin{align*}
		\big\langle R\big(z,T(\tau)\big) (-A^-)^{-\delta}x^-, f^- \big\rangle
		=
		\sum_{n=N_{\rm a}}^{\infty}
		\frac{\langle x^-,\psi_n \rangle \langle \phi_n,f^- \rangle}
		{(-\lambda_n)^{\delta} (z - e^{\lambda_n \tau})} =
		\big\langle x^-, R\big(\overline{z}, 
		T(\tau)^*\big) (-A^-_*)^{-\delta}f^- \big\rangle
	\end{align*}
	for all $x^- \in X^-$, where $N_{\rm a} \in \mathbb{N}$ satisfies
	\eqref{eq:Ns_def}.
	Therefore,
	\begin{align*}
		\big\|F^-R\big(z,T^-(\tau)\big) (-A^-)^{-\delta}\big\|  
		&=
		\sup\big\{ \big\langle x^-, R\big(\overline{z}, 
		T(\tau)^*\big) (-A^-_*)^{-\delta}f^- \big\rangle:
		x^- \in X^- \text{~with~} \|x^-\| = 1 \big\} \\
		&\leq
		\big\|
		R\big(\overline{z}, 
		T(\tau)^*\big) (-A^-_*)^{-\delta}f^-
		\big\|
	\end{align*}
	is obtained.
\end{proof}

We divide the proof of \eqref{eq:RSFR_estimate} and \eqref{eq:RSFR_estimate_poly_decay} into three cases:
(i) $\beta \geq \alpha$; (ii) $\gamma \geq \alpha$; and (iii) $\beta,\gamma <\alpha$,  as in the proof of Lemma~19 in \cite{Paunonen2014JDE}.
For the proof,
we introduce some constants.
Take $\tau >0$, and 
let $N_{\rm b} \in \mathbb{N}$ be such that \eqref{eq:lambda_cond_large}
holds.
Under (A\ref{assump:imaginary}),
there exist constants
$r_1 > 1$ and
$c_1 >0$ such that
\begin{equation}
	\label{eq:unstable_eig_bound}
	|z - e^{\tau \lambda_n} | \geq c_1
\end{equation}
for all $z \in \mathbb{D}_{r_1} \cap \mathbb{E}_1$ and
$1 \leq n \leq  N_{\rm b}-1$. 

\subsubsection{Case $\beta \geq \alpha$}
First, we consider the case $\beta \geq \alpha$.
\begin{lmm}
	\label{lem:g_zero}
	Suppose that 
	{\rm (A\ref{assump:finite_unstable})} and {\rm (A\ref{assump:imaginary})}
	hold. Let
	$b \in \mathcal{D}^{\beta}$ for some $\beta \geq \alpha$ and $f \in X$.
	Then
	\eqref{eq:RSFR_estimate}
	holds for a fixed $\tau >0$.
	Moreover, if $0 < \delta \leq  \alpha /2$, then
	\eqref{eq:RSFR_estimate_poly_decay} holds for a fixed $\tau >0$.
\end{lmm}
\begin{proof}
	Let $\tau >0$ be given.
	Since 
	\[
	\big\|FR\big(z, T(\tau)\big)\big\| = \big\|R\big(\overline{z}, T(\tau)^*\big) f\big\|,
	\]
	for all $z \in \rho(T(\tau))$,
	it follows from Lemma~\ref{lem:Resol_T_int}.a) that 
	\[
	\lim_{r \downarrow 1} ~(r-1)
	\int^{2\pi}_0 \big\|FR\big(re^{i\theta}, T(\tau)\big)\big\|^2  d\theta  = 0.
	\]
	In order to prove \eqref{eq:RSFR_estimate}, it suffices to verify that 
	\begin{equation}
		\label{eq:gamma_zero_bounded}
		\sup_{z \in \mathbb{D}_{r_1} \cap \mathbb{E}_{1}} \big\|R\big(z, T(\tau)\big)S(\tau)\big\|< \infty.
	\end{equation}
	
	Take $z \in \mathbb{D}_{r_1} \cap \mathbb{E}_{1}$.
	Recalling 
	the series expansion \eqref{eq:RS_series} for $R(z,T(\tau)) S(\tau)$, 
	we obtain
	\begin{align*}
		\big\|R\big(z, T(\tau)\big)S(\tau)\big\|^2 \leq 
		M_{\rm  b}
		\sum_{n=1}^\infty
		\left|
		\frac{1 - e^{\tau \lambda_n}}{z - e^{\tau \lambda_n}} \cdot 
		\frac{\langle b ,\psi_n\rangle }{\lambda_n}
		\right|^2.
	\end{align*}
	By (A\ref{assump:finite_unstable}) and (A\ref{assump:imaginary}),
	there is a constant $\kappa >0$ such that $|\lambda_n| \geq \kappa$ 
	for all $n \in \mathbb{N}$.
	By the estimate \eqref{eq:unstable_eig_bound},
	\begin{align}
		\label{eq:N1_smaller_gamma0}
		\sum_{n=1}^{N_{\rm b}-1}
		\left|
		\frac{1 - e^{\tau \lambda_n}}{z - e^{\tau \lambda_n}} \cdot 
		\frac{\langle b ,\psi_n\rangle }{\lambda_n}
		\right|^2 \leq 
		\frac{(1+e^{\tau \sup_{n \in \mathbb{N} } \re \lambda_n })^2}{c_1^2 \kappa^2}
		\sum_{n=1}^{N_{\rm b}-1} \left|
		\langle b ,\psi_n\rangle 
		\right|^2.
	\end{align}
	Lemma~\ref{lem:frac_lam_bound} with $\widetilde \alpha\coloneqq \beta$ shows that
	\begin{align}
		\label{eq:N1_larger_gamma0}
		\sum_{n=N_{\rm b}}^\infty
		\left|
		\frac{1 - e^{\tau \lambda_n}}{z - e^{\tau \lambda_n}} \cdot 
		\frac{\langle b ,\psi_n\rangle }{\lambda_n}
		\right|^2
		&\leq \Upsilon_1^2
		\sum_{n=N_{\rm b}}^\infty 
		\left|
		\langle b ,\psi_n\rangle 
		\right|^2  +  \Upsilon_2^2
		\sum_{n=N_{\rm b}}^\infty 
		|\lambda_n|^{2\beta} \, \left|
		\langle b ,\psi_n\rangle 
		\right|^2 
	\end{align}
	for some constants $\Upsilon_1,\Upsilon_2 >0$ independent of $z$.
	Since 
	$b \in \mathcal{D}^{\beta}$, the inequalities \eqref{eq:N1_smaller_gamma0}
	and \eqref{eq:N1_larger_gamma0} 
	yield
	\eqref{eq:gamma_zero_bounded}.
	
	Let $0 < \delta \leq  \alpha /2$.
	To show the second assertion \eqref{eq:RSFR_estimate_poly_decay}, 
	we observe from the estimate \eqref{eq:unstable_eig_bound} that
	\begin{equation*}
		\left\|R\big(z, T(\tau)^*\big) y^+ \right\|^2 
		\leq \frac{1}{M_{\rm a}}
		\sum_{n = 1}^{N_{\rm a}-1} \left|\frac{ 
			\langle y^+,\phi_n\rangle 
		}{z - e^{\tau \overline{\lambda_n }}} \right|^2 
		\leq 
		\frac{M_{\rm b}}{M_{\rm a}} \cdot 
		\frac{\left\| y^+\right\|^2}{c_1^2}
	\end{equation*}
	for all $y^+ \in X^+_*$ and
	$z \in \mathbb{D}_{r_1} \cap \mathbb{E}_1$.
	By this inequality and Lemma~\ref{lem:F_adjoint}.a), 
	\begin{align*}
		\limsup_{r \downarrow 1} \,\Lambda_{\delta/\alpha}(r) \int_0^{2\pi}
		\big\|F^+R\big(re^{i\theta}, T^+(\tau)\big) \big\|^2 d\theta  &\leq
		\limsup_{r \downarrow 1} \,\Lambda_{\delta/\alpha}(r)
		\int^{2\pi}_0 \big\|R\big(re^{i\theta}, T(\tau)^*\big)f^+ \big\|^2  d\theta \\
		&= 0.
	\end{align*}
	Combining Lemmas~\ref{lem:Resol_T_int}.b) and \ref{lem:F_adjoint}.b),
	we also obtain
	\begin{align*}
		&\limsup_{r \downarrow 1} \,\Lambda_{\delta/\alpha}(r)
		\int^{2\pi}_0
		\big\|F^-R\big(re^{i\theta}, 
		T^-(\tau)\big) (-A^-)^{-\delta} \big\|^2  d\theta \\
		&\qquad \leq
		\limsup_{r \downarrow 1} \,\Lambda_{\delta/\alpha}(r)
		\int^{2\pi}_0 \big\|R\big(re^{i\theta}, 
		T(\tau)^*\big) (-A^-_*)^{-\delta}f^- \big\|^2  d\theta  \\
		&\qquad = 0.
	\end{align*}
	The second assertion \eqref{eq:RSFR_estimate_poly_decay}
	then follows from the estimate \eqref{eq:gamma_zero_bounded}.
\end{proof}

\subsubsection{Case $\gamma \geq \alpha$}
For the case $\gamma \geq \alpha$ and the case $\beta,\gamma < \alpha$, we 
need a preliminary lemma.
Note that 
a constant $\Upsilon_0$ in the next lemma depends on the sampling
period $\tau$ unlike the constants $\Upsilon_1$ and $\Upsilon_2$ in Lemma~\ref{lem:frac_lam_bound}, 
because we consider the situation $r \downarrow 1$
for a fixed $\tau >0$
in Proposition~\ref{prop:RSFR_bound}.
\begin{lmm}
	\label{lem:z_lambda_bound}
	Suppose that  {\rm (A\ref{assump:finite_unstable})} and {\rm (A\ref{assump:imaginary})} hold.
	Fix $\tau >0$ and let $r_1 >1$ satisfy \eqref{eq:unstable_eig_bound}
	for all $z \in \mathbb{D}_{r_1} \cap \mathbb{E}_1$,
	$1 \leq n \leq  N_{\rm b}-1$ and some $c_1 >0$. Then
	there exists a constant $\Upsilon_0 >0$ such that 
	for all $z \in \mathbb{D}_{r_1} \cap \mathbb{E}_1$ and
	$n \in \mathbb{N}$,
	\begin{equation}
		\label{eq:Up_zero}
		\frac{1}{|z-e^{\tau \lambda_n} |\, |\lambda_n|^{\alpha}} \leq \Upsilon_0.
	\end{equation}
\end{lmm}
\begin{proof}
	If (A\ref{assump:finite_unstable}) and (A\ref{assump:imaginary}) hold, then
	we have a constant $\kappa >0$ satisfying $|\lambda_n| \geq \kappa$ 
	for all $n \in \mathbb{N}$.
	Since $r_1> 1$ and  $c_1>0$ are chosen so that 
	\eqref{eq:unstable_eig_bound} holds, it follows that
	\[
	\frac{1}{|z-e^{\tau \lambda_n} |\, |\lambda_n|^{\alpha}}
	\leq \frac{1}{c_1 \kappa^{\alpha}}
	\]
	for all $z \in \mathbb{D}_{r_1} \cap \mathbb{E}_1$
	and $1 \leq n \leq N_{\rm b} - 1$.
	Let $n \geq N_{\rm b}$ and $z \in  \mathbb{E}_1$. 
	We consider the following three cases:
	(i)~$\tau \re \lambda_n \leq -1$; (ii)~$\tau \re \lambda_n > -1$
	and $\re\lambda_n \leq -\omega$; and (iii)~$\tau \re \lambda_n > -1$
	and $\lambda_n \in \Omega_{\alpha,\Upsilon}$,
	as in the proof of 
	Lemma~\ref{lem:frac_lam_bound}.
	Moreover,
	we use the estimate
	\begin{align*}
		\frac{1}{|z-e^{\tau \lambda_n} |\, |\lambda_n|^{\alpha}}
		\leq 
		\frac{1}{(1-e^{\tau \re \lambda_n} )|\lambda_n|^{\alpha}}
		=
		\left|
		\frac{1}{\frac{1-e^{\tau \re \lambda_n}}{\tau \re \lambda_n}}\right| 
		\frac{1}{\tau |\re \lambda_n|\,  |\lambda_n|^{\alpha}}.
	\end{align*}

	In the case~(i) $\tau \re \lambda_n \leq -1$, we have  that 
	\[
	\frac{1}{|1-e^{\tau \re \lambda_n} |\, |\lambda_n|^{\alpha}} \leq \frac{1}{(1-e^{-1})\kappa^\alpha}.
	\]
	We next consider the case~(ii) $\tau \re \lambda_n > -1$
	and $\re\lambda_n \leq -\omega$.
	The mean value theorem for the function $t \mapsto e^t$ on $[-1,0]$
	shows that
	\[
	\frac{1 - e^t}{|t|} \geq e^{-1}
	\]
	for all $t \in [-1,0]$. Substituting $t = \tau \re \lambda_n$, 
	we obtain
	\begin{equation}
		\label{eq:dis_bount2_again}
		\frac{1 - e^{\tau \re \lambda_n}}{\tau |\re \lambda_n|} \geq e^{-1}.
	\end{equation}
	Hence
	\[
	\left|
	\frac{1}{\frac{1-e^{\tau \re \lambda_n}}{\tau \re \lambda_n}}\right| \,
	\frac{1}{\tau |\re \lambda_n|\, |\lambda_n|^{\alpha}} \leq 
	\frac{e}{\tau \omega \kappa^{\alpha}}.
	\]	
	Finally, we study the case~(iii) $\tau \re \lambda_n > -1$
	and $\lambda_n \in \Omega_{\alpha,\Upsilon}$.
	Note that the estimate \eqref{eq:dis_bount2_again} holds also
	in the case~(iii).
	Since $\lambda_n \in \Omega_{\alpha,\Upsilon}$ implies
	\[
	\frac{1}{|\re \lambda_n|} \leq \frac{|\im \lambda_n|^{\alpha}}{\Upsilon},
	\]
	we obtain
	\begin{align*}
		\left|
		\frac{1}{\frac{1-e^{\tau \re \lambda_n}}{\tau \re \lambda_n}}\right| \,
		\frac{1}{\tau |\re \lambda_n|\, |\lambda_n|^{\alpha}} \leq 
		\frac{e}{\tau \Upsilon} \cdot \frac{|\im \lambda_n|^{\alpha}}{|\lambda_n|^{\alpha}} 
		\leq \frac{e}{\tau \Upsilon} .
	\end{align*}
	
	Define
	a constant $\Upsilon_0>0$ by
	\[
	\Upsilon_0 \coloneqq 
	\max\left\{
	\frac{1}{c_1 \kappa^{\alpha}},~
	\frac{1}{(1-e^{-1})\kappa^\alpha},~\frac{e}{\tau \omega 
		\kappa^{\alpha}},~\frac{e}{\tau \Upsilon}
	\right\}.
	\]	
	Then \eqref{eq:Up_zero} holds for all $z \in \mathbb{D}_{r_1} \cap \mathbb{E}_1$ and 
	$n  \in \mathbb{N}$.
\end{proof}

We are now in a position to examine the case $\gamma  \geq \alpha$.
\begin{lmm}
	\label{lem:b_zero}
	Suppose that 
	{\rm (A\ref{assump:finite_unstable})} and 
	{\rm (A\ref{assump:imaginary})} hold. Let
	$b \in X$ and $f \in \mathcal{D}_*^{\gamma}$ for some $\gamma \geq \alpha$.
	Then
	\eqref{eq:RSFR_estimate}  holds for a fixed $\tau >0$.
	Moreover, if $0< \delta \leq \min \{1,\,\alpha/2\}$, then 
	\eqref{eq:RSFR_estimate_poly_decay}
	holds for a fixed $\tau >0$.
\end{lmm}
\begin{proof}
	Take $\tau >0$ arbitrarily, and define
	\[
	b_0\coloneqq \int^\tau_0 T(s) bds.
	\]
	Lemma~\ref{lem:Resol_T_int}.a) implies that 
	\[
	\lim_{r \downarrow 1} ~(r-1)
	\int^{2\pi}_0 \big\|R\big(re^{i\theta}, T(\tau)\big) b_0\big\|^2  d\theta  = 0.
	\]
	To obtain \eqref{eq:RSFR_estimate},
	it suffices to show that 
	\begin{equation}
		\label{eq:beta_zero_bounded}
		\sup_{z \in \mathbb{D}_{r_1} \cap \mathbb{E}_{1}} \big\|FR\big(z, T(\tau)\big)\big\| < \infty.
	\end{equation}
	
	Under (A\ref{assump:finite_unstable}) and (A\ref{assump:imaginary}),
	there exists a constant $\kappa>0$ such that $|\lambda_n| \geq \kappa$ 
	for all $n \in \mathbb{N}$. 
	Hence
	\begin{align*}
		\big\|FR\big(z, T(\tau)\big)\big\|^2 &= 
		\big\|R\big(\overline{z}, T(\tau)^*\big)f\big\|^2 \\
		&\leq 
		\frac{1}{M_{\rm a} } \sum_{n=1}^\infty 
		\left|
		\frac{\langle f, \phi_n \rangle }{\overline{z} - e^{\tau \overline{\lambda_n}}}
		\right|^2 \\
		&\leq 
		\frac{1}{M_{\rm a}\kappa^{2(\gamma-\alpha)}}
		\sum_{n=1}^\infty 
		\frac{|\lambda_n|^{2\gamma}\,  |\langle f, \phi_n \rangle|^2}{|z - e^{\tau \lambda_n}|^2 \, |\lambda_n|^{2\alpha}}
	\end{align*}
	for all $z \in \mathbb{D}_{r_1} \cap \mathbb{E}_{1}$.
	By Lemma~\ref{lem:z_lambda_bound}, there exists
	a constant $\Upsilon_0 >0$ such that 
	for all $z \in \mathbb{D}_{r_1} \cap \mathbb{E}_{1}$,
	\begin{align*}
		\sum_{n=1}^\infty 
		\frac{|\lambda_n|^{2\gamma}\,  |\langle f, \phi_n \rangle|^2}{|z - e^{\tau \lambda_n}|^2 \, |\lambda_n|^{2\alpha}}&\leq 
		\Upsilon_0^2
		\sum_{n=1}^\infty |\lambda_n|^{2\gamma}\,  |\langle f, \phi_n \rangle|^2.
	\end{align*}
	Therefore, 
	we obtain \eqref{eq:beta_zero_bounded} from 
	$f \in \mathcal{D}_*^{\gamma}$.
	
	Let $0< \delta \leq \min \{1,\,\alpha/2\}$.
	Then
	$b_0 \in D(A) \subset \mathcal{D}^{\delta}$.
	Lemma~\ref{lem:Resol_T_int}.b) shows that 
	\[
	\lim_{r \downarrow 1} \Lambda_{\delta/\alpha}(r)
	\int^{2\pi}_0 
	\big\|R\big(re^{i\theta}, T(\tau)\big) b_0\big\|^2  d\theta  = 0.
	\]
	Moreover, 
	for all $z \in \rho(T(\tau))$,
	\begin{align*}
		\big\|F^+R\big(z, T^+(\tau)\big) \big\| &= 
		\Big\|FR\big(z, T(\tau)\big)\big|_{X^+}\Big\| \leq 
		\big\|FR\big(z, T(\tau)\big)\big\|
	\end{align*}
	and
	\begin{align*}
		\big\|F^-R\big(z, T^-(\tau)\big)(-A^-)^{-\delta}\big\|
		&\leq 
		\Big\|FR\big(z, T(\tau)\big)\big|_{X^-} \Big\| 
		\, \big\| (-A^-)^{-\delta}\big\|
		\leq \big\|FR\big(z, T(\tau)\big)\big\| \,
		\big\| (-A^-)^{-\delta}\big\|.
	\end{align*}
	Combining these estimates with \eqref{eq:beta_zero_bounded}
	yields
	\eqref{eq:RSFR_estimate_poly_decay}.
\end{proof}

\subsubsection{Case $\beta,\gamma < \alpha$}
Finally, we consider the case  $\beta,\gamma < \alpha$.
For this case, we use the following
simplified version of the moment inequality. We refer  to
Proposition~6.6.4 of
\cite{Haase2006} and Theorem~II.5.34 of \cite{Engel2000}
for the proof of the moment inequality.
\begin{prpstn}
	\label{prop:moment_inequality}
	Let $A$ be the generator of a uniformly bounded $C_0$-semigroup
	on a Banach space $X$ such that $0 \in \rho(A)$.
	Let $0 < \beta < \alpha$.
	Then there exists a constant $\varsigma >0$ such that
	\[
	\|(-A)^{-\beta} x\| \leq \varsigma \|x\|^{1-\beta/\alpha} \,
	\|(-A)^{-\alpha} x\|^{\beta/\alpha}
	\]
	for all $x \in X$.
\end{prpstn}

In the case $\beta,\gamma < \alpha$,
we prove \eqref{eq:RSFR_estimate} and \eqref{eq:RSFR_estimate_poly_decay}
separately.
\begin{lmm}
	\label{lem:bg_posi1}
	Suppose that 
	{\rm (A\ref{assump:finite_unstable})} and 
	{\rm (A\ref{assump:imaginary})}
	hold. Let $0 \leq \beta,\gamma < \alpha $
	and $\beta + \gamma \geq \alpha$. If $b \in \mathcal{D}^{\beta}$
	and $f \in \mathcal{D}_*^{\gamma}$,
	then
	\eqref{eq:RSFR_estimate}  holds for a fixed $\tau >0$.
\end{lmm}
\begin{proof}
	By assumption, we obtain $\beta ,\gamma >0$.
	There exist
	$0 < \beta_1 \leq \beta$ and $0<\gamma_1 \leq \gamma$
	such that $\beta_1+\gamma_1 =\alpha$.
	Since $b \in \mathcal{D}^{\beta}$
	and $f \in \mathcal{D}_*^{\gamma}$,
	we obtain $b^- \in D((-A^-)^{\beta_1})$ and 
	$f^- \in D((-A_*^-)^{\gamma_1})$.

	Take $\tau >0$ and $z \in \mathbb{D}_{r_1} \cap \mathbb{E}_1 \subset 
	\rho (T(\tau))$, where $r_1> 1$ is chosen so that 
	\eqref{eq:unstable_eig_bound} holds for some $c_1 >0$.
	Define 
	\begin{equation}
		\label{eq:b1_def}
		b_1 \coloneqq \int^\tau_0 T(s)(-A^-)^{\beta_1} b^-ds \in X^-.
	\end{equation}
	Since the resolvent
	$R(z, T^-(\tau))$ and the operator on $X^-$
	\[
	x^- \mapsto
	\int^{\tau}_0 T^-(s) x^-ds 
	\] 
	commute with $(-A^-)^{\beta_1}$
	by Proposition~3.1.1.f) of \cite{Haase2006}, 	it follows that 
	\begin{equation}
		\label{eq:Rb1}
		(-A^-)^{\beta_1} R\big(z, T^-(\tau)\big) \int^{\tau}_0 T^-(s) b^- ds =
		R\big(z, T^-(\tau)\big) b_1.
	\end{equation}
	By the moment inequality given in Proposition~\ref{prop:moment_inequality}, 
	there exists $\varsigma_1>0$ such that 
	\[
	\|(-A^-)^{-\beta_1}x^-\| \leq \varsigma_1 \|x^-\|^{1-\beta_1/\alpha} 
	\, \|(-A^-)^{-\alpha} x^-\|^{\beta_1/\alpha}
	\]
	for all $x^- \in X^-$.
	Applying this inequality to $x^- = R(z, T^-(\tau)) b_1$,
	we have from \eqref{eq:Rb1} that 
	\begin{align*}
		\left\|R\big(z, T(\tau)\big) \int^{\tau}_0 T(s) b^- ds\right\|
		&=
		\left\| (-A^-)^{-\beta_1}
		R\big(z, T^{-}(\tau)\big)  b_1
		\right\| \\
		&\leq 
		\varsigma_1 \left\|
		R\big(z, T(\tau)\big)  b_1
		\right\|^{1-\beta_1/\alpha}  \,
		\left\|
		(-A^-)^{-\alpha} R\big(z, T^-(\tau)\big)  b_1
		\right\|^{\beta_1/\alpha}.
	\end{align*}
	Hence
	\begin{align}
		\big\|R\big(z, T(\tau)\big) S(\tau)\big\|
		&\leq 
		\left\|R\big(z, T(\tau) \big) \int^{\tau}_0 T(s) b^+ds
		\right\| 
		+ \left\|R\big(z, T(\tau)\big) \int^{\tau}_0 T(s) b^- ds\right\| \notag \\
		&\leq 	\left\|R\big(z, T(\tau)\big) \int^{\tau}_0 T(s) b^+ds
		\right\| + \varsigma_1 \left\|
		R\big(z, T(\tau)\big)  b_1
		\right\|^{1-\beta_1/\alpha}  \,
		\left\|
		(-A^-)^{-\alpha} R\big(z, T^-(\tau)\big)  b_1
		\right\|^{\beta_1/\alpha} .
		\label{eq:RS_bound}
	\end{align}
	Using Lemma~\ref{lem:z_lambda_bound}, we obtain
	\begin{align*}
		\left\|
		(-A^-)^{-\alpha} R\big(z, T^-(\tau)\big)  b_1
		\right\|^2 &\leq M_{\rm b} 
		\sum_{n=N_{\rm a}}^{\infty} 
		\frac{|\langle b_1, \psi_n \rangle |^2}
		{|z- e^{\tau \lambda_n}|^2 \, |\lambda_n|^{2\alpha}}\\
		&\leq \frac{M_{\rm b} }{M_{\rm a}} \Upsilon_0^2 \|b_1\|^2
	\end{align*}
	for some $\Upsilon_0 >0$.
	Therefore,
	\begin{equation}
		\label{eq:A_inv_alpha_b}
		\left\|
		(-A^-)^{-\alpha} R\big(z, T^-(\tau)\big)  b_1
		\right\|^{\beta_1/\alpha} \leq  
		\left(
		\sqrt{\frac{M_{\rm b} }{M_{\rm a}}} \Upsilon_0\|b_1\|
		\right)^{\beta_1/\alpha} \eqqcolon \varpi_1 .
	\end{equation}
	Combining the estimates \eqref{eq:RS_bound} and 
	\eqref{eq:A_inv_alpha_b}, we obtain
	\begin{equation}
		\label{eq:RS_bound_final}
		\big\|R\big(z, T(\tau)\big) S(\tau)\big\| \leq 
		\left\|R\big(z, T(\tau)\big) \int^{\tau}_0 T(s) b^+ds
		\right\|  + \varsigma_1\varpi_1 \left\|
		R\big(z, T(\tau)\big)  b_1
		\right\|^{1-\beta_1/\alpha}.
	\end{equation}
	
	Define $f_1 \coloneqq (-A^-_*)^{\gamma_1 }f^- \in X^-_*$.
	Then
	\[
	R\big(\overline{z}, T(\tau)^*\big)  f^- =
	(-A^-_*)^{-\gamma_1 } R\big(\overline{z}, T_*^{-}(\tau)\big)  f_1.
	\]
	By the moment inequality given in Proposition~\ref{prop:moment_inequality}, 
	there exists $\varsigma_2 >0$ such that
	\[
	\left\|R\big(\overline{z}, T(\tau)^*\big)  f^- \right\| \\
	\leq
	\varsigma_2 \left\|R\big(\overline{z}, T(\tau)^*\big) f_1 \right\|^{1-\gamma_1/\alpha} \, 
	\big\|(-A^-_*)^{-\alpha} R\big(\overline{z}, T^-_*(\tau)\big) f_1\big\|^{\gamma_1/\alpha}.
	\]
	Then
	\begin{align*}
		\big\|F R\big(z, T(\tau)\big)\big\| 
		&\leq 
		\left\|R\big(\overline{z}, T(\tau)^*\big)  f^+
		\right\| 
		+ \left\|R\big(\overline{z}, T(\tau)^*\big)  f^- \right\| \\
		&\leq \left\|R\big(\overline{z}, T(\tau)^*\big)  f^+
		\right\| + 
		\varsigma_2 \left\|R\big(\overline{z}, T(\tau)^*\big) f_1 \right\|^{1-\gamma_1/\alpha} \, 
		\big\|(-A^-_*)^{-\alpha} R\big(\overline{z}, T^-_*(\tau)\big) f_1\big\|^{\gamma_1/\alpha}.
	\end{align*}
	As in \eqref{eq:A_inv_alpha_b}, we obtain
	\[
	\big\|(-A^-_*)^{-\alpha} R\big(\overline{z}, T^-_*(\tau)\big) f_1\big\|^{\gamma_1/\alpha} \leq \left(
	\sqrt{\frac{M_{\rm b} }{M_{\rm a}}} \Upsilon_0\|f_1\|
	\right)^{\gamma_1/\alpha} \eqqcolon \varpi_2.
	\]
	Hence
	\begin{equation}
		\label{eq:FR_final}
		\big\|F R\big(z, T(\tau)\big)\big\| \leq
		\left\|R\big(\overline{z}, T(\tau)^*\big)  f^+
		\right\|  + 
		\varsigma_2\varpi_2 \left\|R\big(\overline{z}, T(\tau)^*\big) f_1 \right\|^{1-\gamma_1/\alpha}.
	\end{equation}
	
	Define 
	\[
	p \coloneqq \frac{1}{1 - \frac{\beta_1}{\alpha}},\quad 
	q \coloneqq  \frac{1}{1 - \frac{\gamma_1}{\alpha}}.
	\]
	Then we have from
	$\beta_1 + \gamma_1 = \alpha$ that
	$1/p + 1/q = 1$. 
	Since $(d_1+d_2)^2 \leq 2(d_1^2+d_2^2)$ for all $d_1,d_2 \geq 0$,
	it follows from the estimates 
	\eqref{eq:RS_bound_final} and \eqref{eq:FR_final} that 
	\begin{align*}
		\big\|R\big(z, T(\tau)\big) S(\tau)\big\|^2 \,  \big\|FR\big(z, T(\tau)\big)\big\|^2 &\leq 
		4 \left( \left\|R\big(z, T(\tau)\big) \int^{\tau}_0 T(s) b^+ds
		\right\|^2 +
		(\varsigma_1 \varpi_1)^2\left\|
		R\big(z, T(\tau)\big)  b_1
		\right\|^{2/p}  
		\right) \\
		&\qquad \qquad 
		\times \left( \left\|R\big(\overline{z}, T(\tau)^*\big)  f^+\right\|^2 + 
		(\varsigma_2 \varpi_2)^2 \left\|R\big(\overline{z}, T(\tau)^*\big) f_1 \right\|^{2/q}
		\right).
	\end{align*}
	By H\"older's inequality and Lemma~\ref{lem:Resol_T_int}.a), 
	\begin{align*}
		&\limsup_{r \downarrow 1} \,(r-1)
		\int^{2\pi}_0 \left\|
		R\big(re^{i\theta}, T(\tau)\big)  b_1
		\right\|^{2/p}  \,  \left\|R\big(re^{-i\theta}, T(\tau)^*\big) f_1 
		\right\|^{2/q} d\theta  \\
		&\quad \leq 
		\limsup_{r \downarrow 1} \,\left((r-1) 
		\int^{2\pi}_0 		\left\|
		R\big(re^{i\theta}, T(\tau)\big)  b_1
		\right\|^{2} d\theta \right)^{1/p} \cdot 
		\limsup_{r \downarrow 1} \,\left((r-1) 
		\int^{2\pi}_0 \left\|R\big(re^{i\theta}, T(\tau)^*\big) f_1 \right\|^{2} d\theta \right)^{1/q} \\
		&\quad = 0.
	\end{align*}
	Since the estimate \eqref{eq:unstable_eig_bound} yields
	\begin{align}
		\label{eqResol_x_plus}
		\left\|R\big(z, T(\tau)\big) x^+ \right\|^2 
		\leq M_{\rm b}
		\sum_{n = 1}^{N_{\rm a}-1} \frac{ 
			|\langle x^+,\psi_n\rangle |^2
		}{|z - e^{\tau \lambda_n }|^2} 
		\leq 
		\frac{M_{\rm b}}{M_{\rm a}} \cdot 
		\frac{\left\| x^+\right\|^2}{c_1^2}
	\end{align}
	for all $x^+ \in X^+$,
	it follows that 
	\[
	\limsup_{r \downarrow 1} \,
	(r-1) \int^{2\pi}_0 \left\|R\big(re^{i\theta}, T(\tau)\big) \int^{\tau}_0 T(s) b^+ds
	\right\|^{2p} d\theta= 0.
	\]
	Using H\"older's inequality  and Lemma~\ref{lem:Resol_T_int}.a) again,
	we obtain
	\begin{align*}
		&\limsup_{r \downarrow 1} \,(r-1)
		\int^{2\pi}_0 \left\|R\big(re^{i\theta}, T(\tau)\big) \int^{\tau}_0 T(s) b^+ds
		\right\|^2  \,  \big\|R\big(re^{-i\theta}, T(\tau)^*\big) f_1 \big\|^{2/q} d\theta \\
		&\quad \leq \limsup_{r \downarrow 1} \,
		\left((r-1) \int^{2\pi}_0 \left\|R\big(re^{i\theta}, T(\tau)\big) \int^{\tau}_0 T(s) b^+ds
		\right\|^{2p} d\theta
		\right)^{1/p} \\
		&\qquad \times \limsup_{r \downarrow 1} \,\left((r-1) 
		\int^{2\pi}_0 \big\|R\big(re^{i\theta}, T(\tau)^*\big) f_1 \big\|^{2} d\theta \right)^{1/q} \\
		&\quad =0.
	\end{align*}
	Similarly, 
	\begin{align*}
		\limsup_{r \downarrow 1} \,(r-1)
		\int^{2\pi}_0 \left\|
		R\big(re^{i\theta}, T(\tau)\big)  b_1
		\right\|^{2/p}  \, \big\|R\big(re^{-i\theta}, T(\tau)^*\big)  f^+\big\|^2 d\theta =0
	\end{align*}
	and
	\begin{align*}
		\limsup_{r \downarrow 1} \,(r-1)
		\int^{2\pi}_0 \left\|R\big(re^{i\theta}, T(\tau)\big) \int^{\tau}_0 T(s) b^+ds
		\right\|^2  \,  \big\|R\big(re^{-i\theta}, T(\tau)^*\big)  f^+\big\|^2 d\theta =0.
	\end{align*}
	Thus, the desired conclusion \eqref{eq:RSFR_estimate} is obtained.
\end{proof}
\begin{lmm}
	\label{lem:bg_posi2}
	Suppose that 
	{\rm (A\ref{assump:finite_unstable})} and 
	{\rm (A\ref{assump:imaginary})}
	hold. Let $b \in \mathcal{D}^{\beta}$
	and $f \in \mathcal{D}_*^{\gamma}$
	for some $0 \leq \beta,\gamma < \alpha $
	satisfying $\beta + \gamma \geq \alpha$. 
	If $0< \delta \leq \min \{1,\,\alpha/2\}$, then 
	\eqref{eq:RSFR_estimate_poly_decay}
	holds for a fixed $\tau >0$.
\end{lmm}
\begin{proof}
	Let $\beta_1,\gamma_1 >0$ be as in the proof of Lemma~\ref{lem:bg_posi1}.
	If $0< \delta \leq \min \{1,\,\alpha/2\}$,
	then $b_1$ defined by \eqref{eq:b1_def} satisfies
	$b_1 \in D(A^-)
	\subset D((-A^-)^{\delta}).
	$
	By $f^- \in D((-A^-_*)^{\gamma_1})$, we obtain
	\[
	f_2 \coloneqq (-A^-_*)^{-\delta}(-A^-_*)^{\gamma_1}f^- \in D\big((-A^-_*)^{\delta }\big).
	\]
	Take $\tau >0$ arbitrarily.
	Since
	$ (-A^-_*)^{-\delta }f^- =  (-A^-_*)^{-\gamma_1}f_2$,
	it follows that for all $z \in \rho(T(\tau))$,
	\[
	\big\|R\big(\overline{z}, 
	T(\tau)^*\big) (-A^-_*)^{-\delta}f^-\big\|=
	\big\|R\big(\overline{z}, T(\tau)^*\big)  (-A^-_*)^{-\gamma_1}f_2 \big\|.
	\]
	Lemma~\ref{lem:F_adjoint} yields
	\begin{align*}
		\big\|F^+R\big(z, T^+(\tau)\big) \big\|
		&\leq  \|R\big(\overline{z},T(\tau)^*\big) f^+\| 
	\end{align*}
	and
	\begin{align*}
		\big\|F^-R\big(z, 
		T^-(\tau)\big) (-A^-)^{-\delta} \big\|
		&\leq \big\|R\big(\overline{z}, T(\tau)^*\big)  (-A^-_*)^{-\gamma_1}f_2 \big\|.
	\end{align*}
	for all $z \in \rho(T(\tau))$.
	Therefore, we obtain \eqref{eq:RSFR_estimate_poly_decay}
	by
	arguments similar to those proving \eqref{eq:RSFR_estimate}, i.e.,
	a combination of  the moment inequality, 
	the H\"older's inequality, and
	Lemma~\ref{lem:Resol_T_int}.b).
\end{proof}

\begin{proof}[Proof of Proposition~\Rref{prop:RSFR_bound}]
	The assertion follows immediately from
	Lemmas~\ref{lem:g_zero}, \ref{lem:b_zero}, \ref{lem:bg_posi1}, and
	\ref{lem:bg_posi2}.
\end{proof}

\subsection{Proof of Theorem~\ref{thm:SD_SS}}
Now we are able to prove the main result.
\begin{proof}[Proof of Theorem~\Rref{thm:SD_SS}]
	By Theorem~\ref{lem:resol_T_SF} and the combination of
	Lemmas~\ref{lem:cont_time_trans_func_bound} and 
	\ref{lem:discrete_time_trans_func_bound},
	there exist $M_0>0$ and $\tau^*  >0 $  such that 
	for all 
	$\tau \in (0,\tau^*)$, we obtain $\mathbb{E}_1 \subset \rho (\Delta(\tau))$ and 
	\begin{align*}
		&\left|\frac{1}{1 - FR\big(z, T(\tau)\big) S(\tau) } \right| \leq M_0
	\end{align*}
	for all $z \in \rho (T(\tau)) \cap \overline{\mathbb{E}_1}$.
	Take $\tau \in (0,\tau^*)$.
	By
	(A\ref{assump:finite_unstable}),
	there exists $r_0 >1$ such that
	$re^{i\theta} \in \rho(T(\tau))$ 
	for all $r \in (1,r_0)$ and $\theta \in [0,2\pi)$.
	By the Sherman-Morrison-Woodbury formula given in Lemma~\ref{lem:SMW}, we obtain
	\begin{equation}
		\label{eq:SMW}
		R(re^{i\theta}, T(\tau)+S(\tau)F) x =
		R\big(re^{i\theta}, T(\tau)\big) x +
		\frac{R\big(re^{i\theta}, T(\tau)\big) S(\tau)FR\big(re^{i\theta}, T(\tau)\big)  x }
		{1 - FR\big(re^{i\theta}, T(\tau)\big) S(\tau) }
	\end{equation}
	for all $x \in X$, $r \in (1,r_0)$, and  $\theta \in [0,2\pi)$.
	
	a)
	If we show that
	\begin{subequations}
		\label{eq:resol_estimate_for_PB}
		\begin{align}
			\lim_{r \downarrow 1} ~(r-1)
			\int^{2\pi}_0 
			\big\|R\big(re^{i\theta}, \Delta(\tau)\big)x\big\|^2 d\theta = 0
			\quad &\text{for all $x\in X$ and}
			\label{eq:resol_estimate_for_PBa}\\
			\limsup_{r \downarrow 1} \, (r-1)
			\int^{2\pi}_0 
			\big\|R\big(re^{i\theta}, \Delta(\tau) ^*\big)y\big\|^2 d\theta < \infty
			\quad &\text{for all $y\in X$},
			\label{eq:resol_estimate_for_PBb}
		\end{align} 
	\end{subequations}
	then the discrete semigroup $(\Delta(\tau)^k)_{k \in \mathbb{N}}$ 
	is strongly stable by Theorem~\ref{thm:strong_stability_resol},
	and therefore Proposition \ref{prop:SD_DT}.a) implies that 
	the sampled-data system \eqref{eq:sampled_data_sys} is strongly stable.
	
	Since $(d_1+d_2)^2 \leq 2(d_1^2+d_2^2)$ for all $d_1,d_2 \geq 0$, 
	the Sherman-Morrison-Woodbury formula \eqref{eq:SMW} yields
	\begin{align*}
		\int^{2\pi}_0 
		\|R(re^{i\theta}, T(\tau)+S(\tau)F) x\|^2 d\theta  &\leq
		2 \int^{2\pi}_0 
		\big\|R\big(re^{i\theta}, T(\tau)\big) x\big\|^2 d\theta \notag \\
		&\qquad + 2M_0^2 \|x\|^2 \int^{2\pi}_0 \big\|R\big(re^{i\theta}, T(\tau)\big) S(\tau)\big\|^2 
		\,
		\big\|FR\big(re^{i\theta}, T(\tau)\big)\big\|^2 d\theta 
	\end{align*}
	for all $x \in X$ and  $r \in (1,r_0)$.
	By applying Lemma~\ref{lem:Resol_T_int}.a) and Proposition~\ref{prop:RSFR_bound}.a)
	to
	the first and second terms on the right-hand side of this inequality, respectively,
	we obtain \eqref{eq:resol_estimate_for_PBa} for all $x \in X$.
	A similar calculation shows that \eqref{eq:resol_estimate_for_PBb} holds
	for all $y \in Y$.
	In fact, a stronger result than \eqref{eq:resol_estimate_for_PBb},
	\[
	\lim_{r \downarrow 1} ~(r-1)
	\int^{2\pi}_0 
	\big\|R\big(re^{i\theta}, \Delta(\tau) ^*\big)y\big\|^2 d\theta =0,\qquad 
	y \in X,
	\]
	is obtained from the following estimate:
	\begin{align*}
		\int^{2\pi}_0 
		\|R(re^{i\theta}, T(\tau)+S(\tau)F)^* y\|^2 d\theta  
		&=
		\int^{2\pi}_0 
		\left\|R\big(re^{i\theta}, T(\tau)\big)^* y +
		\left[ \frac{R\big(re^{i\theta}, T(\tau)\big) 
			S(\tau)FR\big(re^{i\theta}, T(\tau)\big) }
		{1 - FR\big(re^{i\theta}, T(\tau)\big) S(\tau) } \right]^* y \right\|^2 d\theta 
		\\
		&\quad \leq
		2 \int^{2\pi}_0 
		\big\|R\big(re^{i\theta}, T(\tau)^*\big) y\big\|^2 d\theta \\
		&\qquad + 2M_0^2 \|y\|^2 \int^{2\pi}_0 
		\big\|R\big(re^{i\theta}, T(\tau)\big) S(\tau)\big\|^2 \,
		\big\|FR\big(re^{i\theta}, T(\tau)\big)\big\|^2 d\theta
	\end{align*}
	for all $y \in X$ and  $r \in (1,r_0)$.
	
	b)
	Let $0 < \delta \leq \alpha /2$, and
	assume that $\delta \leq 1$ or $\beta \geq \alpha$.
	Let
	$x \in \mathcal{D}^{ \delta}$ and define
	$x^+\coloneqq \Pi x$ and $x^- \coloneqq (I-\Pi)x$, where 
	$\Pi$ is the projection operator given in \eqref{eq:projection}. Then
	$x^- \in D((-A^-)^{\delta})$, and hence
	$x^- = (-A^-)^{-\delta} \xi^-$ for some $\xi^- \in X^-$.
	For all $z \in \rho(T(\tau))$, 
	\begin{align*}
		FR\big(z, T(\tau)\big)x&=
		F^+R\big(z, T^+(\tau)\big)x^+ + 
		F^-R\big(z, T^-(\tau)\big)(-A^-)^{- \delta} \xi^-.
	\end{align*}
	Since $(d_1+d_2+d_3)^2 \leq 3(d_1^2+d_2^2+d_3^2)$ for all 
	$d_1,d_2,d_3 \geq 0$,
	the Sherman-Morrison-Woodbury formula \eqref{eq:SMW} yields
	\begin{align*}
		&\int^{2\pi}_0 
		\|R(re^{i\theta}, T(\tau)+S(\tau)F) x\|^2 d\theta  \\
		&\qquad \leq
		3 \int^{2\pi}_0 
		\big\|R\big(re^{i\theta}, T(\tau)\big) x\big\|^2 d\theta 
		+ 3M_0^2 \|x^+\|^2 \int^{2\pi}_0 \big\|R\big(re^{i\theta}, T(\tau)\big) S(\tau)\big\|^2 
		\,
		\big\|F^+R\big(re^{i\theta}, T^+(\tau)\big) \big\|^2 d\theta \\
		&\qquad\qquad  + 3M_0^2 \|\xi^-\|^2 \int^{2\pi}_0 \big\|R\big(re^{i\theta}, T(\tau)\big) S(\tau)\big\|^2 
		\,
		\big\|F^-R\big(re^{i\theta}, T^-(\tau)\big)(-A^-)^{- \delta}\big\|^2 d\theta .
	\end{align*}
	We apply Lemma~\ref{lem:Resol_T_int}.b) to the first term
	on the right-hand side and
	Proposition~\ref{prop:RSFR_bound}.b) to the second and third
	terms. Then we obtain
	\[
	\lim_{r\downarrow 1} ~\Lambda_{\delta/\alpha}(r)
	\int^{2\pi}_0 
	\big\|R\big(re^{i\theta}, \Delta (\tau)\big) x \big\|^2 d\theta  =0.
	\]
	
	We have shown in the proof of a) 
	that the discrete semigroup $(\Delta(\tau)^k)_{k \in \mathbb{N}}$ 
	is strongly stable. 
	Therefore, it is power bounded.
	Proposition~\ref{prop:discrete_poly_decay}.b) implies that 
	for all $x \in \mathcal{D}^{\delta}$,
	\begin{align*}
		\|\Delta(\tau)^k x\| = 
		\begin{cases}
			o(k^{-\delta/\alpha}) & \text{if $0 < \delta < \alpha/2$} \vspace{5pt}\\
			o\left(
			\sqrt{\dfrac{\log k}{k} }
			\right) & \text{if $\delta = \alpha/2$}
		\end{cases}
	\end{align*}
	as $k \to \infty$.
	By Proposition \ref{prop:SD_DT}.b) and the subsequent discussion,
	we conclude that 
	for every initial state $x^0 \in \mathcal{D}^{\delta}$,
	the state $x$ of the sampled-data system \eqref{eq:sampled_data_sys} 
	satisfies
	\begin{equation*}
		\|x(t)\| = 
		\begin{cases}
			o(t^{-\delta/\alpha}) & \text{if $0 < \delta < \alpha/2$} \vspace{5pt}\\
			o \left(
			\sqrt{\dfrac{\log t}{t}}
			\right)& \text{if $\delta = \alpha/2$}
		\end{cases}
	\end{equation*}
	as $t\to \infty$.
\end{proof}
\section{Example}
\label{sec:wave_eq}
Consider the controlled
wave equation with Dirichlet boundary conditions
\begin{equation}
	\label{eq:wave_eq}
	\left\{
	\begin{alignedat}{4}
		\frac{\partial^2 w}{\partial t^2} (\xi,t) &=
		\frac{\partial^2 w}{\partial \xi^2} (\xi,t)
		+b_0(\xi) u(t)
		,\quad 0\leq \xi \leq 1,~t \geq 0 \\
		w(0,t) &= w(1,t) = 0,\quad t \geq 0\\
		w(\xi,0) &= w_0(\xi),~ \frac{\partial w}{\partial t}(\xi,0) = w_1(\xi),\quad 0\leq\xi \leq 1,
	\end{alignedat}
	\right.
\end{equation}
where 
$b_0 \colon [0,1] \to \mathbb{R}$ 
is a shaping function 
around the control point and
$u(t) \in \mathbb{R}$ is the control input at time $t \geq 0$.
First, we write the equation \eqref{eq:wave_eq}  as an abstract evolution equation; see  Example~3.2.16 in \cite{Curtain2020} and Example~VI.8.3 in \cite{Engel2000} for details.
Define the operator $A_0 \colon D(A_0) \subset L^2(0,1) \to L^2(0,1)$
by 
\[
A_0w \coloneqq \frac{d^2w}{d\xi^2}
\] 
with domain
$
D(A_0) \coloneqq H^2(0,1) \cap H^1_0(0,1)= \{ 
w \in H^{2}(0,1) :  w(0)= w(1) = 0
\}
$.
The operator $-A_0$ is self-adjoint and positive definite on $L^2(0,1)$.
Hence there exists  a unique positive definite square root $(-A_0)^{1/2}$
with domain $D((-A_0)^{1/2}) = H^{1}_0(0,1) =
\{ 
w \in H^{1}(0,1) :  w(0)= w(1) = 0
\}.
$
Define the Hilbert space
$X \coloneqq D((-A_0)^{1/2}) \times L^2(0,1)$ endowed with
the inner product 
\[
\langle x, y \rangle \coloneqq
\Big \langle (-A_0)^{1/2}x_1, (-A_0)^{1/2}y_1 \Big\rangle_{L^2} + 
\langle x_2, y_2 \rangle_{L^2},\quad x = 
\begin{bmatrix}
	x_1 \\ x_2
\end{bmatrix} \in X,~y = 
\begin{bmatrix}
	y_1 \\ y_2
\end{bmatrix} \in X.
\]
Let $b_0 \in L^2(0,1)$, and put
\[
b \coloneqq 
\begin{bmatrix}
	0 \\ b_0
\end{bmatrix} \in X.
\]
Define
\[
A_1 \coloneqq 
\begin{bmatrix}
	0 & I \\
	A_0 & 0
\end{bmatrix}
\]
with domain $
D(A_1) \coloneqq D(A_0) \times D((-A_0)^{1/2})
$
and 
\[
Bu \coloneqq 
bu,\quad u \in \mathbb{C}.
\]
Then
the controlled wave equation \eqref{eq:wave_eq} can be written
in the form $\dot x =A_1 x + Bu$, where
\[
x(t) \coloneqq
\begin{bmatrix}
	w(\cdot,t) \vspace{6pt}\\ \dfrac{\partial w}{\partial t} (\cdot, t)\vspace{4pt}
\end{bmatrix},\quad t \geq 0.
\]

We denote by $H^{-1}(0,1)$ the dual of $H^1_0(0,1)$
with respect to the pivot space $L^2(0,1)$.
The duality pairing between $H^1_0(0,1)$ and $H^{-1}(0,1)$ 
is denoted by $\langle g, \nu \rangle_{H^1_0,H^{-1}}$ for
$g \in H^1_0(0,1)$ and $\nu  \in H^{-1}(0,1)$.
Then $A_0$ has a unique extension such that $A_0 \in \mathcal{L}(
H^1_0(0,1), H^{-1}(0,1))$, and this extension is unitary; see Corollary~3.4.6 and
Proposition 3.5.1 of
\cite{Tucsnak2009}.
Let
$\zeta_1 \in H^{-1}(0,1)$ and 
$\zeta_2,\eta_0 \in L^2(0,1)$.
We now consider the perturbed wave equation
\begin{equation}
	\label{eq:wave_eq_perturbed}
	\left\{
	\begin{alignedat}{4}
		\frac{\partial^2 w}{\partial t^2} (\xi,t) &=
		\frac{\partial^2 w}{\partial \xi^2} (\xi,t)
		+b_0(\xi) u(t) 
		+ 
		\left(
		\langle
		w, \zeta_1
		\rangle_{H^1_0,H^{-1}} + 
		\left\langle
		\frac{\partial w}{\partial t}, \zeta_2
		\right\rangle_{L^2} 
		\right)\eta_0(\xi),\quad 0\leq \xi \leq 1,~t \geq 0 \\
		w(0,t) &= w(1,t) = 0,\quad t \geq 0\\
		w(\xi,0) &= w_0(\xi),~ \frac{\partial w}{\partial t}(\xi,0) = w_1(\xi),\quad 0\leq\xi \leq 1.
	\end{alignedat}
	\right.
\end{equation}
Put
\[
\zeta \coloneqq 
\begin{bmatrix}
	-A_0^{-1}\zeta_1 \\ \zeta_2 
\end{bmatrix} \in X,\quad 
\eta \coloneqq
\begin{bmatrix}
	0 \\ \eta_0
\end{bmatrix} \in X,
\]
and define $V \in \mathcal{L}(X)$ by
\[
Vx \coloneqq 
\langle x, \zeta \rangle \eta = 
\begin{bmatrix}
	0 \\
	\big(
	\langle
	x_1, \zeta_1
	\rangle_{H^1_0,H^{-1}} + 
	\left\langle
	x_2, \zeta_2
	\right\rangle_{L^2} 
	\big)\eta_0
\end{bmatrix},
\quad x =
\begin{bmatrix}
	x_1 \\ x_2
\end{bmatrix}\in X,
\]
which  is in the form of
one-rank perturbations. The perturbed 
wave equation \eqref{eq:wave_eq_perturbed} is transformed into
the abstract evolution equation  $\dot x = A x + Bu$, where $A \coloneqq A_1 +V$ with domain $D(A) = D(A_1)$.
Assume that the perturbations $\zeta_1$, $\zeta_2$, and $\eta_0$ are chosen
so that
$A$
is a Riesz-spectral operator of the form \eqref{eq:RS_operator} and 
has
the spectral properties
(A\ref{assump:finite_unstable}) and (A\ref{assump:imaginary}).
Such perturbations $\zeta_1$, $\zeta_2$, and $\eta_0$ can be constructed
with minor modifications of the proof of 
Theorem~13 in \cite{Paunonen2011}; see also
Theorem~1 in \cite{Xu1996}.

We apply the spectral decomposition by the projection $\Pi$ given in \eqref{eq:projection}.
Let $N_{\rm a} \in \mathbb{N}$ 
satisfy \eqref{eq:Ns_def}, and assume that 
$\langle b, \psi_n \rangle \not=0$
for all $n=1,\dots, N_{\rm a} - 1$.
This condition on $b$ is satisfied if and only if
there exists $f_1 \in X^+_* = \Pi^* X$ such that
the matrix
\begin{align} 
	\label{eq:wave_unstable_part}
	\begin{bmatrix}
		\lambda_1 & & 0 \\
		& \ddots &  \\
		0 & & \lambda_{N_{\rm a}-1}
	\end{bmatrix} + 
	\begin{bmatrix}
		\langle b, \psi_1 \rangle \\
		\vdots \\
		\langle b, \psi_{N_{\rm a}-1} \rangle
	\end{bmatrix}
	\begin{bmatrix}
		\langle \phi_1, f_1 \rangle & \cdots & 
		\langle \phi_{N_{\rm a} - 1}, f_1 \rangle
	\end{bmatrix}
\end{align}		 
is Hurwitz; see, e.g.,
Theorem~8.2.3 of \cite{Curtain2020}. 
Choose $f_1 \in X^+_*$ such that the matrix given in \eqref{eq:wave_unstable_part}
is Hurwitz, and
define 
$F_1 \in  \mathcal{L}(X,\mathbb{C})$
by
\[
F_1 x \coloneqq \langle x, f_1\rangle,\quad x \in X.
\] 
Since $(T^-(t))_{t \geq 0}$ is polynomially stable with parameter $\alpha$,
the $C_0$-semigroup $(T_{BF_1}(t))_{t\geq 0}$ 
generated by
$A+BF_1$ has the same stability property
by Theorem~9 of \cite{Paunonen2014OM}.

Let $\beta, \gamma \geq 0$ satisfy $\beta + \gamma >\alpha$.
Assume that $b \in \mathcal{D}^{\beta}$, and
take $f_2 \in \mathcal{D}_*^\gamma$.
We choose
$\widetilde \beta \in [0,\beta)$ and
$\widetilde \gamma \in [0,\gamma)$ such that
$\widetilde \beta + \widetilde \gamma \geq \alpha$.
Since $b \in \mathcal{D}^\beta$ and 
$f_1 \in \mathcal{D}_*^\gamma$,
Lemma~\ref{lem:ABF_domain} implies that
\[
b \in D\Big((-A-BF_1)^{\widetilde \beta}\Big),\quad
f_2 \in D\Big((-A^*-F_1^*B^*)^{\widetilde \gamma}\Big).
\]
Define the feedback operator $F \in \mathcal{L}(X,\mathbb{C})$ by
$Fx \coloneqq \langle x, f_1 + f_2\rangle $ for $x \in X$.
By Theorem~6 of \cite{Paunonen2014JDE}, there exists $c >0$ such that 
$A+BF$ also generates a polynomially stable $C_0$-semigroup 
with parameter $\alpha$
whenever
\begin{equation}
	\label{eq:wave_norm_cond}
	\big\|(-A^*-F_1^*B^*)^{\widetilde \gamma} f_2 \big\| < c.
\end{equation}
Hence, if
$f_1 \in X^+_*$ and $f_2 \in \mathcal{D}^{\gamma}_*$ are chosen so that
the matrix given in \eqref{eq:wave_unstable_part} is Hurwitz and 
the norm condition \eqref{eq:wave_norm_cond} holds, then
Assumption~\ref{assum:for_MR} is satisfied, and 
by Theorem~\ref{thm:SD_SS},
the sampled-data system 
\eqref{eq:sampled_data_sys} 
is strongly stable for all sufficiently small sampling periods.
Moreover, let $\alpha = 2$. Then
$\mathcal{D}^{\alpha/2} = D(A) = D(A_1)$, and
therefore the state $x$ of the sampled-data system \eqref{eq:sampled_data_sys} 
satisfies
\[
\|x(t)\| = o \left(
\sqrt{\frac{\log t}{t}}
\right)\qquad (t \to \infty)
\]
for every 
initial state $x^0 \in D(A_1) = (H^2(0,1) \cap H^1_0(0,1)) \times H^1_0(0,1)$.

\section{Conclusion}
\label{sec:conclusion}
We have studied the robustness of polynomial stability
with respect to sampling.
The generator we consider  is a Riesz-spectral operator
whose eigenvalues may approach the imaginary axis asymptotically.
We have presented conditions for the preservation of strong stability
under fast sampling.
Moreover, an estimate for the rate of decay of the state has been provided for
the sampled-data system with a smooth initial state and
a sufficiently small sampling period.
Future work will focus on 
relaxing the assumption on the generator and addressing
systems with multi- and infinite-dimensional input spaces.

\end{document}